\documentclass[11pt,a4paper]{article}

\usepackage{amsfonts,amsmath,layout}
\usepackage{mathrsfs}  
\usepackage{amssymb,mathabx,amsfonts,amsthm,amscd,stmaryrd,dsfont,esint,upgreek,constants,todonotes}

\usepackage[utf8]{inputenc}
\usepackage{makeidx}
\usepackage[T1]{fontenc}
\usepackage{amsmath}
\usepackage{amsthm}
\usepackage{amsfonts}
\usepackage{amssymb}
\usepackage{alltt}
\usepackage{color}
\usepackage{graphicx}
\usepackage{mathtools}
\mathtoolsset{showonlyrefs}

\newcommand{\N}{\mathbb N}

\newcommand{\Z}{\mathbb Z}

\newcommand{\R}{\mathbb R}

\def\P{\mathbb P}

\def\rg{\rangle}
\def\lg{\langle}
\def\Pw{{\mathcal P}(\T^d)}

\def\dw{{\bf d}_2}

\newcommand{\be}{\begin{equation}}
\newcommand{\ee}{\end{equation}}
\def\1{{\bf 1}}

\def\ep{\epsilon}

\def\Dt0{{\bf D}(t_0)}

\def\E{{\bf E}}

\def\P{{\bf P}}

\def\to{\rightarrow}

\def\dw{{\bf d}_2}

\def\rg{\rangle}
\def\lg{\langle} 
\def \to{\rightarrow}

\def \T{\mathbb{T}}
\def \R {\mathbb{R}}
\def \N {\mathbb{N}}
\def \Z {\mathbb{Z}}

\def\dive{{\rm div}}

\def \ep{\varepsilon}

\def\E{\mathbb E}
\def\P{\mathbb P}
\def\inte{\int_{\T^d}}
\definecolor{ProcessBlue}{cmyk}{1,0,0,0.40}

\def\mPT{{\mathcal P}(\T^d)}

\newtheorem{Theorem}{Theorem}[section]
\newtheorem{Definition}[Theorem]{Definition}
\newtheorem{Proposition}[Theorem]{Proposition}
\newtheorem{Lemma}[Theorem]{Lemma}

\newtheorem{Remark}{Remark}[section]

\begin{document}
\title{Weak KAM theory for potential MFG}
\author{Pierre Cardaliaguet, Marco Masoero \thanks{PSL Research University, Universit\'e Paris-Dauphine, CEREMADE, Place de Lattre de Tassigny, F- 75016 Paris, France}}

\maketitle


\begin{abstract}
We develop the counterpart of weak KAM theory for potential mean field games. This allows to describe the long time behavior of time-dependent potential mean field game systems. Our main result is the existence of a limit, as time tends to infinity, of the value function of an optimal control problem stated in the space of measures. In addition, we show a mean field limit for the ergodic constant associated with the corresponding Hamilton-Jacobi equation.
\end{abstract}

\section*{Introduction}
The theory of mean field games (MFG), introduced simultaneously and independently by Lasry and Lions \cite{lasry2006jeux,lasry2006jeux2} and Huang, Caines and  Malham{\'e} \cite{huang2006large}, is devoted to the analysis of models where a large number of players interact strategically with each other. Under suitable assumptions, the Nash equilibria of those games can be analyzed through the solutions of the, so-called, MFG system
\begin{equation}\label{basicMFG}
\begin{cases} 
-\partial_t u-\Delta u+H(x,Du)=F(x,m) & \mbox{in } \T^d\times[0,T]\\
-\partial_t m+\Delta m+{\rm div}(mD_pH(x,Du))=0 & \mbox{in }\T^d\times[0,T]\\
m(0)=m_0,\; u(T,x)=G(x,m(T))&\mbox{in }\T^{d}.
\end{cases}
\end{equation} 
The first unknown $u(t,x)$ is the value function of an infinitesimal player starting from $x$ at time $t$ while the second one, $m(t)$, describes the distribution of players at time $t$. The maps $F,G:\Pw\to \R$ (where $\Pw$ is the set of Borel probability measures on the torus $\T^d$) describe the interactions between players.
 
  In this paper we investigate the limit behavior, as the horizon $T$ tends to infinity, of this system. This is a very natural question, especially when one looks at those models as dynamical systems. 
  
   One natural guess is that the system simplifies in large times and converges to a time independent model, called the ergodic MFG system:
 \begin{equation}\label{eq.ergosyst}
\begin{cases}
- \bar\lambda-\Delta u+H(x,Du)=F(x,m)&\mbox{in } \T^d,\\
\Delta m+{\rm div}(mD_pH(x,Du))=0 & \mbox{in }\T^d.
\end{cases}
\end{equation}

There is a relatively wide evidence of this phenomenon, starting from \cite{lionsmean} and the Mexican wave model in \cite{gueant2011mean} to more recent contributions in \cite{cardaliaguet2013long,cardaliaguet2013long2,cardaliaguet2012long,gomes2010discrete}. 
  
  All these papers, however, rely on a structure property, the so-called monotonicity assumption, which is seldom met in practice. More recently, the problem of understanding what happens in the non-monotone setting has been addressed in several papers. Gomes and Sedjro \cite{gomes2018onedimensional} found the first example of periodic solutions in the context of one-dimensional first order system with congestion. Cirant in \cite{cirant2019existence} and Cirant and Nurbekyan in \cite{cirant2018variational} forecast and then proved the existence of periodic solutions for a specific class of second order MFG systems (with quadratic Hamiltonian). These periodic trajectories were built through a bifurcation method in a neighborhood of a simple solution. Note that these examples show that the ergodic MFG system is not always the limit of the time-dependent ones. In  \cite{masoero2019} the second authors gave additional evidence of this phenomena using ideas from {\it weak KAM theory} \cite{fathi1997solutions,fathi1997theoreme,fathi1998convergence}. The main interest of the approach is that it allows to study the question for a large class of MFG systems,  {\it potential MFG systems}.

We say that a MFG system like \eqref{basicMFG} is of potential type if it can be derived as optimality condition of the following optimal control problem on the Fokker-Plank equation
\begin{equation}\label{UTUT}
\mathcal U^T(t,m_{0})= \inf_{(m,\alpha)}\int_0^T\int_{\T^d}H^*\left(x,\alpha(s,x)\right)dm(s)+\mathcal F(m(s))dt+\mathcal G(m(T)),
\end{equation}
where $(m,\alpha)$ verifies the Fokker Plank equation $-\partial_{t}m+\Delta m+{\rm div}(\alpha m)=0$ with $m(t)=m_{0}$ and $\mathcal F$ and $\mathcal G$ are respectively the potentials of the functions  $F$ and $G$ that appear in \eqref{basicMFG}. Since the very beginning, this class of models has drawn a lot of attention. \cite{lasry2006jeux2} first explained the mechanism behind the minimizing problem \eqref{UTUT} and the MFG system \eqref{basicMFG} and, since then, the literature on potential MFG thrived. See for instance \cite{briani2018stable,cardaliaguet2015second,ferreira2014convergence, meszaros2018variational} for the use of theses techniques to build solutions and  analyse their long-time behavior under a monotonicity assumption. 

In the present paper, we investigate the behavior, as $T\to+\infty$, of the solutions to the mean field games system \eqref{basicMFG} which are minimizers of \eqref{UTUT}. It is a  continuation of \cite{masoero2019}, which started the analysis of the convergence of the time-dependent, non-monotone, potential MFG systems through weak KAM techniques. We believe that these techniques lead to a more fundamental understanding of long time behavior for potential MFG. When the powerful tools of the weak KAM theory can be deployed, one can look at this problem in a more systematic way. Unlike the PDEs techniques that were so far used, this approach does not depend on the monotonicity of the system. A key point is that the weak KAM theory, exploiting the Hamiltonian structure of potential MFG, gives us a clear understanding of the limit object that the trajectories minimize when the time goes to infinity. We draw fully from both Fathi's seminal papers \cite{fathi1997solutions,fathi1997theoreme,fathi1998convergence} and his book \cite{fathi2008weak}. Several objects defined along the paper and the very structure of many proofs will sound familiar for who is acquainted with weak KAM theory. Nonetheless, it is not always straightforward to transpose these techniques into the framework of MFG and it often requires more effort than in the standard case. It is worthwhile to mention that infinite dimensional weak KAM theorems are not new, especially in the context of Wasserstein spaces: see for instance \cite{gamgbo2010lagrangian,gamgbo2014weak,gomes2015minimizers,gomes2016infinite}. These papers do not address the MFG problem but they surely share the same inspiration. 

Let us now present our main results and discuss the strategy of proofs. As we have anticipated, the starting point of this paper are some results proved in \cite{masoero2019}. The first one is the existence of the ergodic constant $\lambda$, such that 
$$
\frac{\mathcal U^T(0,\cdot)}{T}\longrightarrow -\lambda,
$$
where $\mathcal U^T$ is defined in \eqref{UTUT}. 
The second one is the existence of corrector functions. We say that a continuous function $\chi$ on $\mathcal P(\T^d)$ is a corrector function if it verifies the following dynamic programming principle 
\begin{equation}\label{eq.defchichi}
\chi(m_{0})= \inf_{(m,\alpha)}\left( \int_{0}^{t}H^{*}(x,\alpha)dm(s)+\mathcal F(m(s))ds+\chi (m(t))\right)+\lambda t
\end{equation}
where $(m,\alpha)$ solves in the sense of distributions $-\partial_{t}m+\Delta m+{\rm div}(m\alpha)=0$ with initial condition $m(0)=m_{0}$. At the heuristic level, this amounts to say that $\chi$ solves the ergodic problem
\begin{equation}\label{eq.ergobpPw}
\int_{\T^d} ( H(y, D_m\chi(m,y))-\dive_y D_m\chi(m,y)) m(dy) = \mathcal F(m)+ \lambda\qquad {\rm in}\; \Pw.
\end{equation}
(the notion of derivative $D_m\chi$ is described in Section \ref{sec.assum.res} below).

The main results of this paper are Theorem \ref{teo.xiconv} and Theorem \ref{teo.malphaconv}. The first one states that $\mathcal U^T(0,\cdot)+\lambda T$ uniformly converges to a corrector function while the second one ensures that this convergence does not hold only at the level of minimization problems but also when it comes to optimal trajectories. In particular, Theorem \ref{teo.malphaconv} says that optimal trajectories for $\mathcal U^T(0,\cdot)+\lambda T$ converge to calibrated curves (i.e., roughly speaking, to global minimizers of \eqref{eq.defchichi}). Let us recall that, in \cite{masoero2019}, the second author provides examples in which the calibrated curves stay away from the solutions of the MFG ergodic system \eqref{eq.ergosyst}. In that framework, our result implies that no solution to the MFG system \eqref{basicMFG} obtained as minimizers of \eqref{UTUT} converges to a solution of the MFG ergodic system.

The convergence of $\mathcal U^T(0,\cdot)+\lambda T$ to a corrector is of course the transposition, in our setting, of Fathi's famous convergence result for Hamilton-Jacobi equations \cite{fathi1998convergence}. The basic strategy of proof is roughly the same. Here, the additional difficulty lies in the fact that, in our infinitely dimensional framework, the Hamiltonian in \eqref{eq.ergobpPw} is neither first order nor  ``uniformly elliptic'' (cf. the term in divergence in \eqref{eq.ergobpPw}). 

We overcome this difficulty by introducing two main ideas that we now describe. As in \cite{fathi1998convergence}, we start with further characterizations of the limit value $\lambda$. Let us set
\begin{equation}\label{HJintro}
I:= \inf_{\Phi\in C^{1,1}({\mathcal P}(\T^d))}\sup_{m\in {\mathcal P}(\T^d)} \int_{\T^d} (H(y, D_m\Phi(m,y)) -\mathcal F(m)-\dive_y D_m\Phi(m,y)) m(dy).
\end{equation}
Then, by duality techniques, one can check the following equality (Proposition \ref{prop.dual})
\begin{equation}\label{dualI}
-I = \min_{(\mu, p_1)}\int_{{\mathcal P}(\T^d)} \int_{\T^d}  \left(H^*\Bigl(y, \frac{dp_1}{d m\otimes \mu}\Bigr)+\mathcal F(m)\right) m(dy)\mu(dm), 
\end{equation}
where $(\mu, p_1)$ are closed measures, in the sense that, for any $\Phi\in C^{1,1}(\mathcal P(\T^d))$, 
\begin{equation}
-\int_{{\mathcal P}(\T^d)\times \T^d} D_m\Phi(m,y) \cdot p_1(dm,dy) +\int_{{\mathcal P}(\T^d)\times \T^d} \dive_y D_m\Phi(m,y) m(dy)\mu(dm) =0.
\end{equation}
 We call Mather measure any couple $(\mu,p_1)$ which minimizes the dual problem (on this terminology, see Remark \ref{rem.Mather}).

One key step is to show that $I=\lambda$. While it is easy to prove that $I\geq\lambda$ (Proposition \ref{prop.lambdaI}), the opposite inequality is trickier. One has to construct a smooth subsolution of the ergodic problem \eqref{eq.ergobpPw} in a context where there is no ``classical'' convolution.  The idea is to look at a finite particle system on $(\T^d)^N$. A similar idea was used in \cite{gamgbo2014weak} for first order problems on the $L^2(0,1)-$torus. The main difference with  \cite{gamgbo2014weak} is that,
for first order problems, the particle system is embedded into the continuous one, which is not the case for problems with diffusion. The argument of proof is therefore completely different. We set $(v^N,\lambda^N)\in C^2((\T^d)^N)\times\R$ solution of
$$
-\sum_{i=1}^N \Delta_{x_i} v^N({\bf x}) + \frac{1}{N} \sum_{i=1}^N H(x_i, N D_{x_i}v^N({\bf x}))=\mathcal F( m^N_{\bf x})+ \lambda^N.
$$ 
Note that, in contrast with \cite{gamgbo2014weak}, the constant $\lambda^N$, here, depends on $N$. Our first main idea is to  introduce the smooth function on $\mathcal P(\T^d)$
$$
W^{N}(m):= \int_{(\T^d)^N} v^N(x_1, \dots, x_N)\prod_{i=1}^N m(dx_i)
$$
and to show that it satisfies in $\Pw$
 $$
  -\int_{\T^d} \dive_y D_m W^N(m, y) m(dy) + \int_{\T^d}  H(y, D_mW^N(m,y), m)m(dy)\leq\mathcal F(m) + \lambda^N+o(N),
 $$
which implies that $I\leq \liminf_N \lambda^N$. To conclude that $I=\lambda$ we proved that $\lambda^N\rightarrow\lambda$. The proof of this result is organized in two steps. The first one (Lemma \ref{lemma.bernstein.vN}) is inspired by \cite{lions2005homogenization} and consists in deriving, through Berstein's method, estimates on $v^N$ of the form 
$$
N\sum_{i=1}^N\vert D_{x_i}v^N({\bf x})\vert^2\leq \bar C_1\qquad\forall {\bf x}\in(\T^d)^N
$$
and to prove that $v^N$ uniformly converges to a Lipschitz function $V$ on $\mathcal P(\T^d)$. In the second one (Proposition \ref{prop.lacker}) we adapt the argument of \cite{lacker2017limit}, which is concerned with the connection between optimal control of McKean-Vlasov
dynamics and the limit behavior of large number of interacting controlled state processes, to show that if $\lambda^*$ is an accumulation point of $\lambda^N$ then
\begin{align}\label{}
V(m_{0})= \inf_{(m,\alpha)}\left( \int_{0}^{T}\int_{\T^d}H^{*}(x,\alpha)dm(s)+\mathcal F(m(s))ds+V (m(T))\right)+\lambda^* T,
\end{align}
where $(m,\alpha)$ solves in the sense of distributions $-\partial_{t}m+\Delta m+{\rm div}(m\alpha)=0$ with initial condition $m_{0}$. Consequently $\lambda^*=\lambda$ and so $\lambda^N\rightarrow\lambda$.

The next difficulty is that the Hamiltonian appearing in \eqref{eq.ergobpPw} is  singular (because of the divergence term). This prevents us to say, as in the classical setting, that Mather measures are supported by graphs on which the Hamilton-Jacobi is somehow satisfied.  To overcome this issue, we introduce our second main idea, the notion of ``smooth'' Mather measures (measures supported by ``smooth" probability measures). We prove that limits of minimizers for $\mathcal U^T$ provide indeed ``smooth'' Mather measures and that,  if $(\mu, p_1)$ is a ``smooth" Mather measure then, if we set 
$$
q_1(x,m):= D_a H^*\left(y,  \frac{dp_1}{d m\otimes \mu}(y,m)\right), 
$$
we have, for $\mu-$a.e. $m\in \mPT$, 
\begin{equation}\label{eq.q1mu}
\int_{\T^d} q_1(y,m) \cdot Dm(y)dy +   \int_{ \T^d}  H(y, q_1(y,m))m(dy)=\mathcal F(m) +\lambda.
\end{equation}
 (see Proposition \ref{eq.lem.mup1}).  Note that \eqref{eq.q1mu} is a kind of reformulation of the ergodic equation \eqref{eq.ergobpPw}, in which $q_1=D_m\chi$ and where the divergence term is integrated by parts.  The rest of the proof is more standard and does not bring new difficulties compared to \cite{fathi1998convergence}. \\

Let us briefly describe the organization of the paper. 
In Section \ref{sec.assum.res}, we fix the main notation and assumption and collect the results of \cite{masoero2019} that we sketched above.
Section \ref{sec.dual.prob} and \ref{sec.Nprob} focus on  further characterizations of the limit value $\lambda$. In particular, in Section \ref{sec.dual.prob}, we prove that \eqref{dualI} and $I\geq \lambda$ hold, while Section \ref{sec.Nprob} is devoted to the analysis of the particle system and the proof that $I=\lambda$.  Section \ref{sec.Mather} gives a closer look to Mather measures and explains \eqref{eq.q1mu}. Section \ref{sec.finconv} contains Theorem \ref{teo.xiconv} and Theorem \ref{teo.malphaconv} and their proofs.  \\

{\bf Acknowledgement:} The first author was partially supported by the ANR (Agence Nationale de la Recherche) project
ANR-16-CE40-0015-01. We would like to thank Marco Cirant who carefully read this manuscript and pointed to a blunder in the previous version of Section \ref{sec.Nprob}.

\section{Assumptions and preliminary results}\label{sec.assum.res}
The aim of this preliminary section is twofold. Firstly, we introduce the notation and the assumptions that we will use throughout the paper. Then, we collect some results from \cite{masoero2019} which are the starting point of this work.
\subsection{Notation and assumptions}
We work on the $d-$dimensional flat torus $\T^{d}=\R^{d}/\Z^{d}$ to avoid boundary conditions and to set the problem on a compact domain. 
We denote by $\mathcal P(\T^{d})$ the set of Borel probability measures on $\T^{d}$. This is a compact, complete and separable set when endowed with the $1$-Wasserstein distance $\mathbf d(\cdot,\cdot)$. Let $m$ be a Borel measure over $[t,T]\times\T^d$, with first marginal the Lebesgue measure $ds$ over $[t,T]$, then with $\{m(s)\}_{s\in[t,T]}$ we denote the disintegration of $m$ with respect to $dt$. We will always consider measures $m$ such that $m(s)$ is a probability measure on $\T^d$ for any $s\in[t,T]$.

If $m$ is such a measure, then $L^{2}_{m}([t,T]\times\T^{d})$ is the set of $m$-measurable functions $f$ such that the integral of $|f|^{2}dm(s)$ over $[t,T]\times\T^{d}$ is finite.

We use throughout the paper the notion of derivative for functions defined on $\mathcal P(\T^{d})$ introduced in \cite{cardaliaguet2015master}. We say that $\Phi:\mathcal P(\T^{d})\rightarrow \R$ is $C^{1}$ if there exists a continuous function $\frac{\delta \Phi}{\delta m}:\mathcal P(\T^{d})\times\T^{d}\rightarrow\R$ such that
$$
\Phi(m_{1})-\Phi(m_{2})=\int_{0}^{1}\int_{\T^{d}} \frac{\delta \Phi}{\delta m}((1-t)m_{2}+tm_{1},x)(m_{1}-m_{2})(dx)dt,\qquad\forall m_{1},m_{2}\in\mathcal P(\T^{d}).
$$
As this derivative is defined up to an additive constant, we use the standard normalization
\begin{equation}\label{deriv.conve}
\int_{\T^{d}} \frac{\delta \Phi}{\delta m}(m,x)m(dx)=0.
\end{equation}
We recall that, if $\mu$, $\nu\in\mathcal P(\T^d)$, the $1$-Wasserstein distance is defined by
\begin{equation}\label{RKd2}
\mathbf d(\mu,\nu)= \sup \left\{ \left. \int_{\T^{d}} \phi(x) \,d (\mu - \nu) (x) \right| \mbox{continuous } \phi : \T^{d} \to \mathbb{R},\, \mathrm{Lip} (\phi) \leq 1 \right\}.
\end{equation}
\textbf{Assumptions:} Throughout the paper the following conditions will be in place.
\begin{enumerate}
	\item $H:\T^d\times\R^d\rightarrow\R$ is of class $C^2$, $p\mapsto D_{pp}H(x,p)$ is Lipschitz continuous, uniformly with respect to $x$. Moreover, there exists $\bar C>0$ that verifies 
	\begin{equation}\label{1}
	\bar C^{-1}I_d\leq D_{pp}H(x,p)\leq \bar CI_d, \quad\forall (x,p)\in\T^d\times\R^d
	\end{equation}
	and $\theta\in(0,1)$, $C>0$ such that the following conditions hold true:
	\begin{align}\label{DxH.bounds}
	\vert D_xH(x,p)\vert \leq C+C\vert p\vert\quad\forall (x,p)\in\T^d\times\R^d
	\end{align}
	and
	\begin{equation}\label{2}
	|D_{xx}H(x,p)|\leq C(1+|p|)^{1+\theta},\quad|D_{x,p}H(x,p)|\leq C(1+|p|)^\theta,\quad\forall (x,p)\in\T^d\times\R^d.
	\end{equation}

	\item $\mathcal F:\mathcal P(\T^{d})\rightarrow\R$ is of class $C^2$. Its derivative $F:\T^d\times\mathcal P(\T^d)\rightarrow\R$ is twice differentiable in $x$ and $D^2_{xx}F$ is bounded.
	Examples of non monotone coupling functions which verify such conditions can be found in \cite{masoero2019}. 
	
\end{enumerate}
Note that some of the above assumptions will not be used explicitly in this paper but have been used in \cite{masoero2019} to prove results that we will assume to hold true. (The only differences are assumption \eqref{DxH.bounds} and that here we need $\mathcal F$ to be of class $C^2$ while in \cite{masoero2019} $\mathcal F$ was required only to be $C^1$).

Very often in the text, we do not need to work explicitly on $\mathcal F$, so, in order to have a lighter notation, we incorporate $\mathcal F$ in the Hamiltonian defining, for any $(x,p,m)\in\T^d\times\R^d\times\mathcal P(\T^d)$,
\begin{equation}\label{def.calH}
\mathcal H(x,p,m):=H(x,p)-\mathcal F(m).
\end{equation}
We denote with $\mathcal H^*$ the Fenchel conjugate of $\mathcal H$ with respect to the second variable. Then, for any $(x,a,m)\in\T^d\times\R^d\times\mathcal P(\T^d)$,
$$
\mathcal H^*(x,a,m)=H^*(x,a)+\mathcal F(m).
$$
We can now introduce the standard minimization problem in potential MFG:
\begin{equation}\label{VlI}
\mathcal U^T(t,m_0)=\inf_{(m,\alpha)}\int_t^T \int_{\T^d}H^*(x,\alpha)dm(s)+\mathcal F(m(s))dt,\quad m_0\in\mathcal P(\T^d)
\end{equation}
where $m\in C^0([t,T],\mathcal P(\T^d))$, $\alpha\in L^2_{m}([t,T]\times\T^d,\R^d)$ and the following equation is verified in the sense of distributions
\begin{equation}\label{fp}
\begin{cases}
-\partial_t m +\Delta m+{\rm div}(m\alpha)=0&\mbox{in }[t,T]\times\T^{d}\\
m(t)=m_0&\mbox{in }\T^{d}.
\end{cases}
\end{equation}

\subsection{Corrector functions and the limit value $\lambda$} 

Here we collect the results already proved in \cite{masoero2019} that we will use. A most important one is the following.
\begin{Theorem}The function \label{uconv}$\frac{1}{T}\mathcal U^T(t,\cdot)$ uniformly converges to a limit value $-\lambda$ when $T$ goes to $+\infty$.
\end{Theorem}

The second result that we will use is the existence of corrector functions. 
\begin{Definition}We say that $\chi:\mathcal{P}(\T^d)\rightarrow\R$ is a corrector function if, for any $m_{0}\in\mathcal P(\T^{d})$ and any $t>0$,
	\begin{equation}\label{dynchi}
	\chi(m_{0})= \inf_{(m,\alpha)}\left( \int_{0}^{t}H^{*}(x,\alpha)dm(s)+\mathcal F(m(s))ds+\chi (m(t))\right)+\lambda t,
	\end{equation}
	where $m\in C^0([0,t],\mathcal P(\T^d))$, $\alpha\in L^2_{m}([0,t]\times\T^d,\R^d)$ and the pair $(m,\alpha)$ solves in the sense of distributions $-\partial_{t}m+\Delta m+{\rm div}(m\alpha)=0$ with initial condition $m_{0}$.
\end{Definition}
 \begin{Proposition}\label{remlip} 
	The set of corrector functions is not empty and uniformly Lipschitz continuous. In addition, if $\chi$ is a continuous map on $\mathcal{P}(\T^d)$ for which equality \eqref{dynchi} holds for some constant $\lambda'$ and for any $t>0$, then $\lambda'=\lambda$.
\end{Proposition}

A last notion that will come at hand is the one of calibrated curve:
\begin{Definition} We say that $(\bar m,\bar \alpha)$, which satisfies \eqref{fp} for any $t\in\R$, is a calibrated curve if there exists a corrector function $\chi:\mathcal P(\T^{d})\rightarrow \R$ such that $(\bar m,\bar\alpha)$ is optimal for $\chi$: for any $t_{1}<t_{2}\in\R$
	\begin{equation}\label{kajehbzredf}
	\chi(\bar m(t_1))=\lambda(t_2-t_1)+ \int_{t_1}^{t_2} \inte H^{*}\left(x, \bar\alpha(s)\right)d\bar m(s) +\mathcal F(\bar m(s))ds+\chi(\bar m(t_2)).
	\end{equation}
	
\end{Definition}

The set of calibrated curves verifies the following property.
\begin{Proposition}\label{duplc}The set of calibrated curves is not empty. Moreover, if $(m,\alpha)$ is a calibrated curve, then $m\in C^{1,2}(\R\times\T^{d})$ and there exists a function $u\in C^{1,2}(\R\times\T^{d})$ such that $\alpha=D_{p}H(x,Du)$ where $(u,m)$ solves
	\begin{equation}
	\begin{cases} 
	-\partial_t u-\Delta u+H(x,Du)=F(x,m) & \mbox{in } \R\times\T^{d},\\
	-\partial_t m+\Delta m+{\rm div}(mD_pH(x,Du))=0 & \mbox{in }\R\times\T^{d}.
	\end{cases}
	\end{equation}
\end{Proposition}
\section{A dual problem}\label{sec.dual.prob}

In this section we introduce the two usual characterizations of the constant $\lambda$: $\lambda$ is expected to be the smallest constant for which there exists a smooth sub-corrector and $-\lambda$ is the smallest value of the Lagrangian when integrated against suitable ``closed'' measures.  The goal of this section is to show that both problems are in duality and have the same value $I$. We postpone the analysis of the equality $I=\lambda$ to the next section. 

We start with the HJ equation which we write in variational form. 
\begin{equation}\label{pbI}
I:= \inf_{\Phi\in C^{1,1}({\mathcal P}(\T^d))}\sup_{m\in {\mathcal P}(\T^d)} \int_{\T^d} (\mathcal H(y, D_m\Phi(m,y),m)-\dive_y D_m\Phi(m,y)) m(dy),
\end{equation}
where by $C^{1,1}({\mathcal P}(\T^d))$ we mean the set of maps $\Phi:{\mathcal P}(\T^d)\to \R$ such that $D_m\Phi$ and $D_yD_m\Phi$ are continuous.  We recall that ${\mathcal H}$ is defined in \eqref{def.calH}. Let us start with a comparison between $I$ and $\lambda$.

\begin{Proposition}\label{prop.lambdaI} We have $I\geq \lambda$. 
\end{Proposition}

\begin{proof} Let $\ep>0$ and $\Phi$ be such that 
$$
\int_{\T^d} (\mathcal H(y, D_m\Phi(m,y),m)-\dive_y D_m\Phi(m,y)) m(dy) \leq I+\ep\qquad \forall m\in \Pw. 
$$
Let $(\bar m,\bar \alpha)$ be a calibrated curve and $\chi$ be a corrector function associated with $(\bar m,\bar \alpha)$. By definition of calibrated curve, $(\bar m,\bar \alpha)$ verifies
$$
-\partial_t \bar m(t)+\Delta \bar m(t)+\dive(\bar \alpha(t) \bar m(t))=0,\qquad\forall t\in\R
$$
and
\begin{equation}\label{chiddp5}
\chi(\bar m(0))= \int_0^T \int_{\T^d} \mathcal H^*(y, \bar \alpha(t,y),\bar m(t))\bar m(t,dy) +\lambda T + \chi(\bar m(T)).
\end{equation}
As $\Phi$ is smooth, we get

\begin{align*}
&\Phi(\bar m(T))-\Phi(\bar m(0))=\int_0^T\frac{d}{dt} \Phi(\bar m(t))dt=\int_0^T\int_{ \T^d} \frac{\delta\Phi}{\delta m}(\bar m(t),y)\partial_t\bar m(t)dydt \\
&= \int_0^T\int_{\T^d} \dive_yD_m\Phi(\bar m(t),y)\bar m(t,dy)dt - \int_0^T\int_{\T^d} D_m\Phi(\bar m(t),y)\cdot \bar \alpha(t,y) \bar m(t,dy)dt \\ 
& \geq -(I+\ep)T + \int_0^T\int_{\T^d} (\mathcal H(y, D_m\Phi(\bar m(t),y),\bar m(t))-D_m\Phi(\bar m(t),y)\cdot \bar \alpha(t,y))\bar m(t,dy)dt\\ 
& \geq -(I+\ep)T - \int_0^T\int_{\T^d} \mathcal H^*(y, \bar \alpha(t,y),\bar m(t))\bar m(t,dy). 
\end{align*}
Using \eqref{chiddp5}, we end up with 
\begin{align*}
\Phi(\bar m(T))-\Phi(\bar m(0)) &  \geq -(I+\ep)T - \int_0^T\int_{\T^d} \mathcal H^*(y, \bar \alpha(t,y),\bar m(t))\bar m(t,dy) \\
& =  -(I+\ep) T+\chi(\bar m(T)) -  \chi(\bar m(0))+\lambda T.
\end{align*}
Using the fact that $\chi$ and $\Phi$ are bounded, we divide both sides by $T$ and we conclude that $\lambda \leq I$ by letting $T\to +\infty$ and $\ep\to 0$ . 
\end{proof}

Next we reformulate $I$ in terms of "closed measures". 

\begin{Proposition}\label{prop.dual} We have 
\begin{equation}\label{dualpb}
-I = \min_{(\mu, p_1)}\int_{{\mathcal P}(\T^d)} \int_{\T^d} \mathcal H^*\Bigl(y, \frac{dp_1}{d m\otimes \mu},m\Bigr) m(dy)\mu(dm), 
\end{equation}
where the minimum is taken over $\mu\in {\mathcal P}(\mPT)$ and $p_1$, Borel vector measure on $\mPT\times \T^d$, such that $p_1$ is absolutely continuous with respect to the measure $dm\otimes \mu:=m(dy)\mu(dm)$ and such that $(\mu, p_1)$ is closed, in the sense that, for any $\Phi\in C^{1,1}(\mathcal P(\T^d))$, 
\begin{equation}\label{ehkjrbndkfgc}
 -\int_{{\mathcal P}(\T^d)\times \T^d} D_m\Phi(m,y) \cdot p_1(dm,dy) +\int_{{\mathcal P}(\T^d)\times \T^d} \dive_y D_m\Phi(m,y) m(dy)\mu(dm) =0.
\end{equation}
\end{Proposition} 

In analogy with weak KAM theory, we call a measure $(\mu,p_1)$ satisfying \eqref{ehkjrbndkfgc} a closed measure and a minimum of \eqref{dualpb} a Mather measure. 

\begin{proof} As usual we can rewrite $I$ as
$$
I=  \inf_{\Phi\in C^{1,1}({\mathcal P(\T^d)})}\sup_{\mu \in {\mathcal P}({\mathcal P}(\T^d))}\int_{{\mathcal P}(\T^d)} \left(\int_{\T^d} (\mathcal H(y, D_m\Phi(m,y),m)-\dive_y D_m\Phi(m,y)) m(dy)\right)\mu(dm). 
$$
We claim that
\begin{equation}\label{Iminmax}
I=  \max_{\mu \in {\mathcal P}({\mathcal P}(\T^d))} \inf_{\Phi\in C^{1,1}({\mathcal P})}\int_{{\mathcal P}(\T^d)} \left(\int_{\T^d} (\mathcal H(y, D_m\Phi(m,y),m)-\dive_y D_m\Phi(m,y)) m(dy)\right)\mu(dm).
\end{equation}
Indeed, ${\mathcal P}({\mathcal P}(\T^d))$ is a compact subspace of $\mathcal M({\mathcal P}(\T^d))$ and, for any fixed $\Phi\in C^{1,1}({\mathcal P(\T^d)})$ the function on $\mathcal M({\mathcal P}(\T^d))$ defined by
$$
\mu\mapsto\int_{{\mathcal P}(\T^d)} \left(\int_{\T^d} (\mathcal H(y, D_m\Phi(m,y),m)-\dive_y D_m\Phi(m,y)) m(dy)\right)\mu(dm)
$$
is continuous and concave (as it is linear). On the other hand, when we fix $\mu\in{\mathcal P}({\mathcal P}(\T^d))$, the function on $C^{1,1}({\mathcal P}(\T^d))$, defined by
$$
\Phi\mapsto\int_{{\mathcal P}(\T^d)} \left(\int_{\T^d} (\mathcal H(y, D_m\Phi(m,y),m)-\dive_y D_m\Phi(m,y)) m(dy)\right)\mu(dm),
$$
is continuous with respect to the uniform convergence in $C^{1,1}({\mathcal P}(\T^d))$ and convex due to the convexity of $\mathcal H$. Therefore, the hypothesis of Sion's min-max Theorem are fulfilled and \eqref{Iminmax} holds true. 

We now define the continuous linear map  $\Lambda: C^{1,1}({\mathcal P}(\T^d))\to (C^0({\mathcal P}(\T^d)\times \T^d))^d\times C^0({\mathcal P}(\T^d)\times \T^d)$ by
$$
\Lambda (\Phi)= (D_m\Phi, \dive_y D_m\Phi).
$$
From now on we fix a maximizer $\mu$ for \eqref{Iminmax} and we define $E:= C^{1,1}({\mathcal P}(\T^d))$, $F:=(C^0({\mathcal P}(\T^d)\times \T^d))^d\times C^0({\mathcal P}(\T^d)\times \T^d)$ and 
$$
f(\Phi)= 0, \; g(a,b)= \int_{{\mathcal P}(\T^d)} \int_{\T^d} (\mathcal H(y, a(m,y),m)-b(m,y)) m(dy)\mu(dm), \qquad \forall (a,b)\in F. 
$$
We note that 
$$
I= \inf_{\Phi\in C^{1,1}({\mathcal P}(\T^d))} \{ f(\Phi)+ g(\Lambda \Phi)\}.
$$

To use the Fenchel-Rockafellar theorem we need to check the transversality conditions. These hypothesis are easily verified, indeed, both $f$ and $g$ are continuous and, therefore, proper functions. The function $f$ is convex because it is linear and so is $g$, due to the convexity of Hamiltonian $\mathcal H$. Moreover, it comes directly from its definition that $\Lambda$ is a bounded linear functional on $E$. Then, the Fenchel-Rockafellar Theorem states that
$$
I= -\min_{(p_1,p_2)\in F'} \{ f^*(-\Lambda^*(p_1,p_2))+g^*(p_1,p_2)\}.
$$
Note that 
$$
F'=({\mathcal M}({\mathcal P}(\T^d)\times \T^d))^d\times {\mathcal M}({\mathcal P}(\T^d)\times \T^d),
$$
that $f^*(q)= 0$ if $q=0$, $f^*(q)=+\infty$ otherwise. So, for any $(p_1,p_2)\in F'$,  
\begin{align*}
& f^*(-\Lambda (p_1,p_2))   = \sup_{\Phi\in C^{1,1}({\mathcal P}(\T^d))}  - \lg \Lambda^*(p_1,p_2),\Phi\rg - f(\Phi)= \sup_{\Phi\in C^{1,1}({\mathcal P}(\T^d))}  - \lg (p_1,p_2),\Lambda (\Phi)\rg \\
&  = \sup_{\Phi\in C^{1,1}({\mathcal P}(\T^d))} -\int_{({\mathcal P}(\T^d)\times \T^d} D_m\Phi(m,y) \cdot p_1(dm,dy) -\int_{({\mathcal P}(\T^d)\times \T^d} \dive_y D_m\Phi(m,y) p_2(dm,dy) 
\end{align*}
which is $0$ if, for any $\Phi\in C^{1,1}({\mathcal P}(\T^d))$, 
$$
\int_{{\mathcal P}(\T^d)\times \T^d} D_m\Phi(m,y) \cdot p_1(dm,dy) +\int_{{\mathcal P}(\T^d)\times \T^d} \dive_y D_m\Phi(m,y) p_2(dm,dy) =0, 
$$
and $+\infty$ otherwise. On the other hand, 
\begin{align*}
g^*(p_1,p_2)& = \sup_{(a,b)\in F} \int_{{\mathcal P}(\T^d)\times \T^d} a(m,y)\cdot p_1(dm,dy)+\int_{{\mathcal P}(\T^d)\times \T^d} b(m,y) p_2(dm,dy)\\
 & \qquad -  \int_{{\mathcal P}(\T^d)} \left(\int_{\T^d} (\mathcal H(y, a(m,y),m)-b(m,y)) m(dy)\right)\mu(dm).
\end{align*}
So, if $g^*(p_1,p_2)$ is finite, one must have that $p_2(dm,dy)=- m(dy)\mu(dm)$ and that $p_1$ is absolutely continuous with respect to the measure $dm\otimes \mu:= m(dy)\mu(dm)$. Indeed, if $p_1$ were not absolutely continuous with respect to $dm\otimes \mu$, we could find a sequence of continuous functions $a_n\in (C^0({\mathcal P}(\T^d)\times \T^d))^d$ such that, $a_n$ is uniformly bounded on the support of $dm\otimes \mu$ and 
$$
\int_{{\mathcal P}(\T^d)\times \T^d} a_n(m,y)\cdot p_1(dm,dy)\rightarrow +\infty.
$$
But then we would have $g^*(p_1,p_2)=+\infty$. So,
\begin{align*}
g^*(p_1,p_2)& =\sup_{a\in C^0({\mathcal P}(\T^d)\times \T^d))^d} \int_{{\mathcal P}(\T^d)\times \T^d} \left(a(m,y)\cdot \frac{dp_1}{d m\otimes \mu}(m,y)-\mathcal H(y, a(m,y),m)\right)m(dy)\mu(dm)
\end{align*}
We now want to prove that
\begin{align*}
&\sup_{a\in C^0({\mathcal P}(\T^d)\times \T^d))^d} \int_{{\mathcal P}(\T^d)\times \T^d} \left(a(m,y)\cdot \frac{dp_1}{d m\otimes \mu}(m,y)-\mathcal H(y, a(m,y),m)\right)m(dy)\mu(dm)\\
&\qquad =  \int_{{\mathcal P}(\T^d)} \int_{\T^d} \mathcal H^*\Bigl(y, \frac{dp_1}{d m\otimes \mu}(m,y),m\Bigr) m(dy)\mu(dm).
\end{align*}
We have one inequality by definition of Fenchel's conjugate. Indeed,
\begin{align}\label{ineenne}
&\sup_{a\in C^0({\mathcal P}(\T^d)\times \T^d))^d} \int_{{\mathcal P}(\T^d)\times \T^d} \left(a(m,y)\cdot \frac{dp_1}{d m\otimes \mu}(m,y)-\mathcal H(y, a(m,y),m)\right)m(dy)\mu(dm)\\
&\leq \int_{{\mathcal P}(\T^d)\times \T^d} a^*(m,y)\cdot \frac{dp_1}{d m\otimes \mu}(m,y)-\mathcal H(y,a^*(m,y),m)m(dy)\mu(dm)\\
&=  \int_{{\mathcal P}(\T^d)} \int_{\T^d} \mathcal H^*\Bigl(y, \frac{dp_1}{d m\otimes \mu}(m,y),m\Bigr) m(dy)\mu(dm),
\end{align}
where 
$$
a^*(m,y)=D_a\mathcal H^*\left(y,\frac{dp_1}{d m\otimes \mu}(m,y),m\right).
$$
For the opposite inequality we use a density argument. The function $a^*$ could be not continuous but yet it must be measurable. Moreover, the growth of $\mathcal H$ ensures that $a^*\in L^2_\mu({\mathcal P}(\T^d)\times \T^d;\R^d)$. As ${\mathcal P}(\T^d)\times \T^d$ is a compact Hausdorff space, the set of continuous functions is dense in $L^2_\mu({\mathcal P}(\T^d)\times \T^d)$. Let $a_n\in C^0{\mathcal P}(\T^d)\times \T^d;\R^d)$ be such that $a_n\rightarrow a^*$ in $L^2_\mu$. Then,
$$
\sup_{a\in C^0({\mathcal P}(\T^d)\times \T^d))^d} \int_{{\mathcal P}(\T^d)\times \T^d} a(m,y)\cdot \frac{dp_1}{d m\otimes \mu}(m,y)-\mathcal H(y, a(m,y),m)m(dy)\mu(dm)\geq
$$
$$
\lim_{n\rightarrow+\infty} \int_{{\mathcal P}(\T^d)\times \T^d} a_n(m,y)\cdot \frac{dp_1}{d m\otimes \mu}(m,y)-\mathcal H(y, a_n(m,y),m)m(dy)\mu(dm)=
$$
$$
\int_{{\mathcal P}(\T^d)\times \T^d} a^*(m,y)\cdot \frac{dp_1}{d m\otimes \mu}(m,y)-\mathcal H(y,a^*(m,y),m)m(dy)\mu(dm)=
$$
$$
  \int_{{\mathcal P}(\T^d)} \int_{\T^d} \mathcal H^*\Bigl(y, \frac{dp_1}{d m\otimes \mu}(m,y),m\Bigr) m(dy)\mu(dm).
$$
Therefore, we can conclude that
\begin{align}
I& = - \min_{(\mu, p_1)}\int_{{\mathcal P}(\T^d)} \int_{\T^d} \mathcal H^*\Bigl(y, \frac{dp_1}{d m\otimes \mu},m\Bigr) m(dy)\mu(dm), 
\end{align}
where the minimum is taken over $(\mu,p_1)$ satisfying condition \eqref{ehkjrbndkfgc}. 
\end{proof}

\section{The $N-$particule problem.}\label{sec.Nprob}

In the previous section, we introduced two problems in duality. These problems have a common value called $I$ and we have checked that $\lambda \leq I$. The aim of this section is to show that there is actually an equality: $\lambda=I$. In the standard setting, this equality is proved by smoothing correctors by a convolution; by the convexity of the Hamiltonian, the smoothened corrector is a subsolution to the corrector equation (up to a small error term), thus providing a competitor for problem \eqref{pbI}. In our framework, there is no exact equivalent of the convolution. We overcome this difficulty by considering the projection of the problem onto the set of empirical measures of size $N$ (thus on $(\T^d)^N$). For the $N-$particle problem, the corrector is smooth. We explain here that a suitable extension of this finite dimensional corrector to the set ${\mathcal P}(\T^d)$  provides a smooth sub-corrector for the problem in ${\mathcal P}(\T^d)$ when $N$ is large. This shows the claimed equality and, in addition, the fact that the ergodic constant associated with the $N-$particle problem converges to $\lambda$.

More precisely, we consider $v^N:(\T^d)^N\rightarrow\R$ the solution of: 
\begin{equation}\label{vN.critic.eq}
-\sum_{i=1}^N \Delta_{x_i} v^N({\bf x}) + \frac{1}{N} \sum_{i=1}^N \mathcal H(x_i, N D_{x_i}v^N({\bf x}), m^N_{\bf x})=\lambda^N, 
\end{equation}
where ${\bf x}= (x_1, \dots, x_N)\in (\T^d)^N$ and $m^N_{\bf x}= N^{-1}\sum_{i=1}^N \delta_{x_i}$ and where ${\mathcal H}$ is defined in \eqref{def.calH}. Let us recall that such a corrector exists (it is unique up to additive constants) and is smooth. To fix the ideas, we choose the solution $v^N$ such that 
\begin{equation}\label{condvN0}
\int_{(\T^d)^N} v^N= 0.
\end{equation}

\begin{Proposition} \label{prop.IleqlambdaN}
	We have 
	$$
	I\leq \liminf_{N\to+\infty} \lambda^N. 
	$$
\end{Proposition}

\begin{proof}
	We define  
	$$
	W^{N}(m):= \int_{(\T^d)^N} v^N(x_1, \dots, x_N)\prod_{i=1}^N m(dx_i). 
	$$
	As $v^N$ is smooth, it is clear that $W^{N}$ is also smooth on ${\mathcal P}(\T^d)$ and we have
	$$
	D_m W^N(m, y) = \sum_{k=1}^N \int_{(\T^d)^{N-1}} D_{x_k} v^N(x_1, \dots, x_{k-1}, y, x_{k+1}, \dots, x_N)\prod_{i\neq k} m(dx_i)
	$$
	and
	$$
	\dive_y D_m W^N(m, y) = \sum_{k=1}^N \int_{(\T^d)^{N-1}} \Delta_{x_k} v^N(x_1, \dots, x_{k-1}, y, x_{k+1}, \dots, x_N)\prod_{i\neq k} m(dx_i).
	$$
	In view of the convexity of $\mathcal H$ with respect to $p$, we obtain, for any $m\in \Pw$, 
	\begin{align*}
	& -\int_{\T^d} \dive_y D_m W^N(m, y) m(dy) + \int_{\T^d} \mathcal H(y, D_mW^N(m,y), m)m(dy) \\
	& = - \sum_{k=1}^N \int_{\T^d} \int_{(\T^d)^{N-1}} \Delta_{x_k} v^N(x_1, \dots, x_{k-1}, y, x_{k+1}, \dots, x_N)\prod_{i\neq k} m(dx_i)m(dy) \\
	& \qquad + \int_{\T^d} \mathcal H\Bigl(y, \sum_{k=1}^N \int_{(\T^d)^{N-1}} D_{x_k} v^N(x_1, \dots, x_{k-1}, y, x_{k+1}, \dots, x_N)\prod_{i\neq k} m(dx_i) , m\Bigr)m(dy) \\ 
	& \leq   \int_{(\T^d)^{N}}  \Bigl( - \sum_{k=1}^N \Delta_{x_k} v^N(x_1,  \dots, x_N) + \frac{1}{N} \sum_{k=1}^N \mathcal H\Bigl(x_k, ND_{x_k} v^N(x_1, \dots, x_N) , m\Bigr)\Bigr) \prod_{i} m(dx_i).
	\end{align*}
	Following \cite{carmona2018probabilistic}, the Glivenko-Cantelli Theorem states that 
	$$
	\int_{(\T^d)^N} \dw(\mu^N_{{\bf x}}, m )  \prod_{i} m(dx_i) \leq \ep_N := \left\{\begin{array}{ll}
	N^{-1/2} & {\rm if }\; d<4,\\ 
	N^{-1/2}\ln(N)& {\rm if }\; d=4,\\ 
	N^{-2/d} &  {\rm otherwise}.
	\end{array}\right.  
	$$
	As $\mathcal H$ has a separate form: $\mathcal H(x,p,m)= H(x,p)-\mathcal F(m)$ where $\mathcal F$ is Lipschitz continuous with respect to $m$, we infer that 
	\begin{align*}
	&  \int_{(\T^d)^{N}} \frac{1}{N} \sum_{k=1}^N \mathcal H\Bigl(x_k, ND_{x_k} v^N(x_1, \dots, x_N) , m\Bigr) \prod_{i} m(dx_i)\\ 
	& \qquad \leq \int_{(\T^d)^{N}} \frac{1}{N} \sum_{k=1}^N \mathcal H\Bigl(x_k, ND_{x_k} v^N(x_1, \dots, x_N) , m^N_{\bf x} \Bigr) \prod_{i} m(dx_i)
	+ C\ep_N. 
	\end{align*}
	Recalling the equation satisfied by $v^N$, we conclude that  
	\begin{align}
	& -\int_{\T^d} \dive_y D_m W^N(m, y) m(dy) + \int_{\T^d} \mathcal H(y, D_mW^N(m,y), m)m(dy) \notag \\
	& \leq  \int_{(\T^d)^{N}}  \Bigl( - \sum_{k=1}^N \Delta_{x_k} v^N(x_1,  \dots, x_N) + \frac{1}{N} \sum_{k=1}^N \mathcal H\Bigl(x_k, ND_{x_k} v^N(x_1, \dots, x_N) , m^N_{\bf x})\Bigr) \prod_{i} m(dx_i) + C\ep_N \notag \\
	& \qquad 
	\leq \lambda^N+  C\ep_N.  \label{eq.WN}
	\end{align}
	As $W^N$ is smooth, this shows that $I\leq \liminf_N \lambda^N$. 
\end{proof}

According to the above proposition and Proposition \ref{prop.lambdaI} we have that $\lambda\leq I\leq \liminf_N \lambda^N$. Therefore, to have that $I=\lambda$ we need to show that $\lim_N\lambda^N=\lambda$. Before proving this result we first need to introduce some estimates $v^N$.
\begin{Lemma}\label{lemma.bernstein.vN}
	There exists a constant $\bar C_1>0$, independent of $N$, such that
	\begin{align}\label{vN.est.conclusion}
	N\sum_{i=1}^N\vert D_{x_i}v^N({\bf x})\vert^2\leq \bar C_1\qquad\forall {\bf x}\in(\T^d)^N.
	\end{align}
\end{Lemma}
\begin{proof}The proof is an application of the Berstein's type estimates used in \cite{lions2005homogenization}.
	First we show the sequence $\lambda^N$ is bounded. This is a direct consequence of the maximum principle. Indeed, if ${\bf y}$ is a minimum point of $v^N$, then 
	\begin{align}\label{}
	\frac{1}{N} \sum_{i=1}^N \mathcal H(x_i, 0, m^N_{\bf y})\geq-\sum_{i=1}^N \Delta_{x_i} v^N({\bf y}) + \frac{1}{N} \sum_{i=1}^N \mathcal H(x_i, 0, m^N_{\bf y})=\lambda^N.
	\end{align}
	As the left hand side uniformly bounded with respect to $N$ we get that $\lambda ^N$ is bounded from above. If instead of the minimum we consider a maximum we get the lower bound.

	From \eqref{vN.critic.eq} and the quadratic growth of the Hamiltonian we have
	\begin{align}\label{}
	&\bar C N\sum_{i=1}^N\vert D_{x_i}v^N({\bf x})\vert^2-\bar C\leq \sum_{i=1}^N \Delta_{x_i} v^N({\bf x})+\Vert\mathcal F \Vert_\infty+\lambda^N\leq \sum_{i,j=1}^N\vert D^2_{x_i,x_j} v^N({\bf x})\vert+\Vert\mathcal F \Vert_\infty+\lambda^N\\
	&\leq N^{\frac{1}{2}}\left(\sum_{i,j=1}^N\vert D^2_{x_i,x_j} v^N({\bf x})\vert^2\right)^{\frac{1}{2}}+\Vert\mathcal F \Vert_\infty+\lambda^N.
	\end{align}
	As $\lambda^N$ is bounded, there exists a constant $C$ such that
	\begin{align}\label{estimate.vNcriteq}
	\bar C N\sum_{i=1}^N\vert D_{x_i}v^N({\bf x})\vert^2\leq N^{\frac{1}{2}}\left(\sum_{i,j=1}^N\vert D^2_{x_i,x_j} v^N({\bf x})\vert^2\right)^{\frac{1}{2}}+C
	\end{align}
	Now we define
	$$
	w({\bf x})=\sum_{i=1}^N \vert D_{x_i}v^N({\bf x})\vert ^2,
	$$
	then
	\begin{align}\label{}
	D_{x_j}w({\bf x})= 2\sum_{i=1}D_{x_i}v^N({\bf x})D_{x_j,x_j}^2v^N({\bf x}).
	\end{align}
	By a direct computation we get
	\begin{align}\label{}
	&-\sum_{i=1}^N\Delta_{x_i}w({\bf x})=-2\sum_{i,j=1}^N\vert D^2_{x_i,x_j}v^N({\bf x})\vert^2+\sum_{i=1}^N D_{x_i}v^N({\bf x}) D_{x_i}\left(\sum_{j=1}^N\Delta_{x_j}v^N({\bf x})\right)\\
	&=-2\sum_{i,j=1}^N\vert D^2_{x_i,x_j}v^N({\bf x})\vert^2+\sum_{i=1}^ND_{x_i}v^N({\bf x})D_{x_i}\left(N^{-1}\left(\sum_{j=1}^N H(x_j,ND_{x_j}v^N({\bf x}))-\mathcal F(m_{\bf x}^N)\right)-\lambda^N\right)
	\end{align}
	\begin{align}
	&=-2\sum_{i,j=1}^N\vert D^2_{x_i,x_j}v^N({\bf x})\vert^2+N^{-1}\sum_{i=1}^ND_{x_i}v^N({\bf x})D_{x}H(x_i,ND_{x_i}v^N({\bf x}))\\
	&\qquad\qquad\qquad+\sum_{i,j=1}^ND_pH(x_j,N\Delta_{x_j}v^N({\bf x}))D^2_{x_i,x_j}v^N({\bf x})D_{x_i}v^N({\bf x})-N^{-1}\sum_{i=1}^ND_xF(m_{x_i}^N,x_i)D_{x_i}v^N({\bf x})\\
	&=-2\sum_{i,j=1}^N\vert D^2_{x_i,x_j}v^N({\bf x})\vert^2+N^{-1}\sum_{i=1}^ND_{x_i}v^N({\bf x})D_{x}H(x_i,ND_{x_i}v^N({\bf x}))\\
	&\qquad\qquad\qquad+\sum_{j=1}^ND_pH(x_j,N\Delta_{x_j}v^N({\bf x}))D_{x_j}w({\bf x})-N^{-1}\sum_{i=1}^ND_xF(m_{\bf x}^N,x_i)D_{x_i}v^N({\bf x})
	\end{align}
	
	If ${\bf x}$ is a maximum point of $w$ the above computations lead to
	\begin{align}\label{}
	2\sum_{i,j=1}^N\vert D^2_{x_i,x_j}v^N({\bf x})\vert^2&\leq N^{-1}\left[\sum_{i=1}^ND_{x_i}v^N({\bf x})D_{x}H(x_i,ND_{x_i}v^N({\bf x}))-\sum_{i=1}^ND_xF(m_{\bf x}^N,x_i)D_{x_i}v^N({\bf x})\right]\\
	&\leq N^{-1}\left[\sum_{i=1}^N\vert D_{x_i}v^N({\bf x})\vert\vert D_{x}H(x_i,ND_{x_i}v^N({\bf x}))\vert+\Vert D_xF\Vert_\infty\sum_{i=1}^N\vert D_{x_i}v^N({\bf x})\vert\right]\\
	&\leq N^{-1}\left[CN\sum_{i=1}^N\vert D_{x_i}v^N({\bf x})\vert^2+(\Vert D_xF\Vert_\infty+C)\sum_{i=1}^N\vert D_{x_i}v^N({\bf x})\vert\right],
	\end{align}
	where for the last line we used the bounds on $D_x H$ \eqref{DxH.bounds}.
	
	Plugging the above inequality into \eqref{estimate.vNcriteq}, as $D_xF$ is bounded, we get (for a possible different constant $C$, independent of $N$, that might change from line to line)
	\begin{align}\label{}
	N\sum_{i=1}^N\vert D_{x_i}v^N({\bf x})\vert^2&\leq C\left[N\sum_{i=1}^N\vert D_{x_i}v^N({\bf x})\vert^2+\sum_{i=1}^N\vert D_{x_i}v^N({\bf x})\vert\right]^{\frac{1}{2}}+C\\
	&\leq C \left(N\sum_{i=1}^N\vert D_{x_i}v^N({\bf x})\vert^2\right)^{\frac{1}{2}}+C \left(N\sum_{i=1}^N\vert D_{x_i}v^N({\bf x})\vert^2\right)^{\frac{1}{4}}+C.
	\end{align}
	The above inequalities ensure that there exists a $\bar C_1>0$, independent of $N$, such that 
	\begin{align}\label{}
	N\sum_{i=1}^N\vert D_{x_i}v^N({\bf x})\vert^2\leq \bar C_1.
	\end{align}
	
\end{proof}

\begin{Lemma}\label{lemma.vN.conv.V}
	Let $v^N$ be the solution to \eqref{vN.critic.eq} satisfying condition \eqref{condvN0}.
	There exists a Lipschitz continuous map $V:\mathcal P(\T^d)\rightarrow\R$ and a subsequence $\{N_k\}_k$ such that 
	\begin{align}\label{}
	\lim_{k\rightarrow+\infty}\sup_{{\bf x}\in(\T^d)^{N_k}}\vert v^{N_k}({\bf x})-V(m^{N_k}_{\bf x})\vert=0
	\end{align}
\end{Lemma}
\begin{proof}
	The proof is a direct application of a compactness result due to Lions (see \cite[Theorem 2.1]{cardaliaguet2010notes}), for which we just need to check the hypotheses. Let us note that, by uniqueness (up to a constant) of the solution to \eqref{vN.critic.eq}, $v^N$ is symmetrical. To apply \cite[Theorem 2.1]{cardaliaguet2010notes} we have to prove that $v^{N}$ is uniformly bounded and that there exists a constant $C$, independent of $N$, such that, for any ${\bf x},{\bf y}\in(\T^d)^N$
	\begin{align}\label{lerhbrsdxc}
	\vert v^N({\bf x})-v^N({\bf y})\vert\leq C {\bf d}(m^N_{\bf x},m^N_{{\bf y}}).
	\end{align}
	
	 The fact that \eqref{lerhbrsdxc} and the normalization condition \eqref{condvN0} hold implies that the $v^N$ are uniformly bounded. For proving \eqref{lerhbrsdxc}, the second hypothesis we fix two points ${\bf x}=(x_1,\cdots,x_N)\in(\T^d)^N$ and ${\bf y}=(y_1,\cdots,y_N)\in(\T^d)^N$. As $v^N$ is symmetrical, we can assume without loss of generality that ${\bf d}_2(m^N_{\bf x},m^N_{{\bf y}})= (N^{-1}\sum_{i=1}^N {\bf d}^2_{\T^d}(x_i,y_i))^{1/2}$, where ${\bf d}_{\T^d}$ is the distance on ${\T^d}$. Then, using again \eqref{vN.est.conclusion}, we have
	\begin{align}\label{}
	\vert v^N({\bf x})-v^N({\bf y})\vert&\leq\sup_{{\bf z}\in(\T^d)^N}\sum_{i=1}^N\vert D_{x_i}v^N({\bf z})\vert{\bf d}_{\T^d}(x_i,y_i)\\
	&\leq\sup_{{\bf z}\in(\T^d)^N}\left(\sum_{i=1}^N\vert D_{x_i}v^N({\bf z})\vert^2\right)^\frac{1}{2}\left(\sum_{i=1}^N{\bf d}^2_{\T^d}(x_i,y_i)\right)^\frac{1}{2}\\
	&\leq\bar C_1^\frac{1}{2}N^{-\frac{1}{2}}\left(\sum_{i=1}^N{\bf d}_{\T^d}^2(x_i,y_i)\right)^\frac{1}{2}=\bar C_1^\frac{1}{2}{\bf d}_2(m^N_{\bf x},m^N_{{\bf y}})\leq C{\bf d}(m^N_{\bf x},m^N_{{\bf y}}).
	\end{align}
	
\end{proof}

We now adapt the argument of \cite{lacker2017limit} to prove that the function $V$ is a corrector.
\begin{Proposition}\label{prop.lacker}
	Let $\{\lambda^{N_k}\}_k$ be the subsequence such that $(v^{N_k})$ converges to $V$ as in Lemma \ref{lemma.vN.conv.V} and let $\lambda^*$ be an accumulation point of $\{\lambda^{N_k}\}_k$. Then, for any $T>0$ and any $m_0\in\mathcal P(\T^d)$, the function $V:\mathcal P(\T^d)\rightarrow\R$ verifies
	\begin{align}\label{}
	V(m_{0})= \inf_{(m,\alpha)}\left( \int_{0}^{T}\int_{\T^d}H^{*}(x,\alpha)dm(s)+\mathcal F(m(s))ds+V (m(T))\right)+\lambda^* T,
	\end{align}
	where $(m,\alpha)$ solves in the sense of distributions $-\partial_{t}m+\Delta m+{\rm div}(m\alpha)=0$ with initial condition $m_{0}$.
	
	Consequently, $\lambda^*=\lambda$ and $\lambda^N\rightarrow\lambda$.
\end{Proposition}
\begin{proof} The Proposition is a consequence of a nice result due to Lacker \cite{lacker2017limit}, which roughly says that the mean field limit of an optimal stochastic control problem involving symmetric controllers is an  optimal control problem of a McKean-Vlasov equation. As the adaptation to our framework, if relatively easy, requires cumbersome details, we refer the interested reader to the appendix.

The fact that $\lambda^*=\lambda$ is a consequence of the uniqueness of the ergodic constant as stated in Proposition \ref{remlip}. As any accumulation point of the bounded sequence $\lambda^N$ is equal to $\lambda$, we finally conclude that the whole sequence converges to $\lambda$.
\end{proof}

As we mentioned before, the immediate consequence of the above proposition is the following result.
\begin{Proposition}\label{prop.I=lambda}
	We have
	$$
	I=\lambda=\lim_{N\rightarrow+\infty}\lambda^N.
	$$
\end{Proposition}

\section{On the support of the Mather measures}\label{sec.Mather}
In this section we take a closer look at Mather measures and the properties of their support points.
\begin{Definition} \label{def.mup1smooth} We say that the closed measure $(\mu, p_1)$ is smooth if there exists a constant $C>0$ such that, for $\mu-$a.e. $m\in \mPT$, $m$ has a positive density and 
$$
\|D\ln(m)\|_{L^2_m(\T^d)}\leq C. 
$$
\end{Definition}

The aim of this section is to prove the following property of smooth Mather measures. 

\begin{Proposition}\label{eq.lem.mup1} Let $(\mu,p_1)$ be a smooth Mather measure. Let us set 
$$
q_1(x,m):= D_a\mathcal H^*\left(y,  \frac{dp_1}{d m\otimes \mu}(y,m),m\right). 
$$
Then we have, for $\mu-$a.e. $m\in \mPT$, 
$$
\int_{\T^d} q_1(y,m) \cdot Dm(y)dy +   \int_{ \T^d} \mathcal H(y, q_1(y,m), m)m(dy)=\lambda.
$$
\end{Proposition}

In order to prove the proposition, let us start with a preliminary step. Let $(\mu, p_1)$ be optimal in problem \eqref{dualpb} (where we recall that $I=\lambda$) and let $\Phi^N$ be any minimizing sequence in \eqref{pbI}. 

\begin{Lemma}\label{lem.cvDmPhi} The sequence
$(D_p\mathcal H(y, D_m\Phi^N(m,y) ))$ converges to $\frac{dp_1}{d m\otimes \mu} $ in  $L^2({\mathcal P}(\T^d)\times\T^d, dm\otimes \mu)$.  Moreover,
\begin{equation}\label{ekzjrsd}
\begin{aligned}
\lim_{N\to+\infty}-\int_{{\mathcal P}(\T^d)\times \T^d} \dive_y D_m \Phi^N(m, y) m(dy)\mu(dm) +   \int_{{\mathcal P}(\T^d)\times \T^d} \mathcal H(y, D_m\Phi^N(m,y), m)m(dy) \mu(dm)\\
=\lambda.
\end{aligned}
\end{equation}
\end{Lemma}

\begin{proof} Recall that $\Phi^N$ satisfies 
\begin{align}\label{eq.WNbis}
& -\int_{\T^d} \dive_y D_m \Phi^N(m, y) m(dy) + \int_{\T^d} \mathcal H(y, D_m\Phi^N(m,y), m)m(dy) 
\leq \lambda+  o_N(1) .
\end{align}
 We integrate equation \eqref{eq.WNbis} against $\mu$ and add the problem for $(\mu,p_1)$ to find: 
\begin{align*}
& -\int_{{\mathcal P}(\T^d)\times \T^d} \dive_y D_m \Phi^N(m, y) m(dy)\mu(dm) + \int_{{\mathcal P}(\T^d)\times \T^d} \mathcal H(y, D_m\Phi^N(m,y), m)m(dy) \mu(dm)\\
& \qquad + \int_{{\mathcal P}(\T^d)\times\T^d}\mathcal  H^*\Bigl(y, \frac{dp_1}{d m\otimes \mu},m\Bigr) m(dy)\mu(dm) \leq   o_N(1) .
\end{align*}
Using the uniform convexity of $\mathcal H$, this implies that
\begin{align*}
& -\int_{{\mathcal P}(\T^d)\times \T^d} \dive_y D_m \Phi^N(m, y) m(dy)\mu(dm) + \int_{{\mathcal P}(\T^d)\times\T^d} \frac{dp_1}{d m\otimes \mu}\cdot  D_m\Phi^N(m,y)  m(dy)\mu(dm) \\
& \qquad  + C^{-1} \int_{{\mathcal P}(\T^d)\times\T^d} \left |\frac{dp_1}{d m\otimes \mu} - D_p\mathcal H(y, D_m\Phi^N(m,y),m )\right|^2m(dy)\mu(dm)   \leq   o_N(1) .
\end{align*}
Then, \eqref{ehkjrbndkfgc} implies that the first line vanishes and so 
\begin{align*}
& \qquad  \int_{{\mathcal P}(\T^d)\times\T^d} \left |\frac{dp_1}{d m\otimes \mu} - D_p\mathcal H(y, D_m\Phi^N(m,y) ,m)\right|^2m(dy)\mu(dm)   \leq   o_N(1) ,
\end{align*}
which proves the first statement of the lemma. We now turn to \eqref{ekzjrsd}. First of all, as $\Phi^N$ is a minimizing sequence for \eqref{pbI}, we have that 
\begin{equation}\label{limsup1.4}
\begin{aligned}
\limsup_{N\to+\infty}-\int_{{\mathcal P}(\T^d)\times \T^d} \dive_y D_m \Phi^N(m, y) m(dy)\mu(dm) +   \int_{{\mathcal P}(\T^d)\times \T^d} \mathcal H(y, D_m\Phi^N(m,y), m)m(dy) \mu(dm)\\
\leq\lambda.
\end{aligned}
\end{equation}
To prove the other inequality we start with
$$
-\int_{{\mathcal P}(\T^d)\times \T^d} \dive_y D_m \Phi^N(m, y) m(dy)\mu(dm) +   \int_{{\mathcal P}(\T^d)\times \T^d} \mathcal H(y, D_m\Phi^N(m,y), m)m(dy) \mu(dm).
$$
We add and subtract the same quantity to get
$$
-\int_{{\mathcal P}(\T^d)\times \T^d} \dive_y D_m \Phi^N(m, y) m(dy)\mu(dm) + \int_{{\mathcal P}(\T^d)\times\T^d} \frac{dp_1}{d m\otimes \mu}\cdot  D_m\Phi^N(m,y)  m(dy)\mu(dm)+
$$
$$
  \int_{{\mathcal P}(\T^d)\times \T^d} \mathcal H(y, D_m\Phi^N(m,y), m)m(dy) \mu(dm)-\int_{{\mathcal P}(\T^d)\times\T^d} \frac{dp_1}{d m\otimes \mu}\cdot  D_m\Phi^N(m,y)  m(dy)\mu(dm).
$$
As $(\mu,p_1)$ verifies \eqref{ehkjrbndkfgc}, the first line above vanishes. Then,  using the Fenchel's inequality, we find that
$$
-\int_{{\mathcal P}(\T^d)\times \T^d} \dive_y D_m \Phi^N(m, y) m(dy)\mu(dm) +   \int_{{\mathcal P}(\T^d)\times \T^d} \mathcal H(y, D_m\Phi^N(m,y), m)m(dy) \mu(dm)\geq
$$
$$
-\int_{{\mathcal P}(\T^d)} \int_{\T^d} \mathcal H^*\Bigl(y, \frac{dp_1}{d m\otimes \mu},m\Bigr) m(dy)\mu(dm).
$$
By hypothesis, $(\mu,p_1)$ is a minimizer for \eqref{dualpb}, so, for any $N\in\N$,
\begin{equation}
-\int_{{\mathcal P}(\T^d)\times \T^d} \dive_y D_m \Phi^N(m, y) m(dy)\mu(dm) +   \int_{{\mathcal P}(\T^d)\times \T^d} \mathcal H(y, D_m\Phi^N(m,y), m)m(dy) \mu(dm)\geq\lambda.
\end{equation}
Therefore,
\begin{equation}\label{liminf1.4}
\begin{aligned}
\liminf_{N\to+\infty}-\int_{{\mathcal P}(\T^d)\times \T^d} \dive_y D_m \Phi^N(m, y) m(dy)\mu(dm) +   \int_{{\mathcal P}(\T^d)\times \T^d} \mathcal H(y, D_m\Phi^N(m,y), m)m(dy) \mu(dm)\\
\geq\lambda.
\end{aligned}
\end{equation}
and the result follows.
\end{proof}

\begin{proof}[Proof of Proposition \ref{eq.lem.mup1}] From our assumption on $(\mu,p_1)$, we have 
\begin{align*}
\int_{{\mathcal P}(\T^d)\times \T^d} \dive_y D_m \Phi^N(m, y) m(dy)\mu(dm) = 
-\int_{{\mathcal P}(\T^d)\times \T^d} D_m \Phi^N(m, y)\cdot Dm(y)dy\mu(dm).
\end{align*}
As, by Lemma \ref{lem.cvDmPhi}, the sequence
$(D_p\mathcal H(y, D_m\Phi^N(m,y) ),m)$ converges to $\frac{dp_1}{d m\otimes \mu} $ in  $L^2({\mathcal P}(\T^d)\times\T^d, dm\otimes \mu)$, we also have that 
$(D_m\Phi^N(m,y) )$ converges to $q_1$  in  $L^2({\mathcal P}(\T^d)\times\T^d, dm\otimes \mu)$ by regularity and invertibility of $D_p\mathcal H$ and $D_a\mathcal H^*$. On the other hand,  $\frac{Dm(y)}{m(y)}$ is bounded in $L^2_m(\T^d)$ for $\mu-$a.e. $m\in\mPT$. Therefore 
\begin{align*}
\lim_{N\to+\infty} \int_{{\mathcal P}(\T^d)\times \T^d} \dive_y D_m \Phi^N(m, y) m(dy)\mu(dm) & = -
\int_{{\mathcal P}(\T^d)\times \T^d} q_1 (m, y)\cdot Dm(y)dy\mu(dm). 
\end{align*}
We conclude thanks to \eqref{ekzjrsd} that 
\begin{align*}
& \int_{{\mathcal P}(\T^d)\times \T^d} q_1 (m, y)\cdot Dm(y)dy\mu(dm)+ \int_{{\mathcal P}(\T^d)\times \T^d} \mathcal H(y, q_1 (m, y) , m)m(dy) \mu(dm)=\lambda.
\end{align*}
On the other hand, extracting a subsequence if necessary, the sequence $\{D_m\Phi^N(m,y)\}_N$ converges $\mu-$a.e. to $q_1$. So, by \eqref{eq.WNbis}, we have, for $m-$a.e. $m\in \mPT$, 
\begin{align*}
& \int_{\T^d} q_1(y,m) \cdot Dm(y)dy + \int_{\T^d} \mathcal H(y, q_1(y,m), m)m(dy) 
\leq \lambda. 
\end{align*}
Putting together the previous inequality with the previous equality gives the result. 

\end{proof}

\section{The long time behavior of potential MFG}\label{sec.finconv}
In this section, we prove the two main results of the paper: the first one is the convergence, as $T\to +\infty$, of ${\mathcal U}^T(0, \cdot)+\lambda T$. The second one states that limits of time-dependent minimizing mean field games equilibria, as the horizon tends to infinity, are calibrated curves.


\subsection{Convergence of $\mathcal U^T(0,\cdot)+\lambda T$}

We recall that $\mathcal U^T(t,m_0)$ is defined by
$$
{\mathcal U}^T(t,m_0)= \inf_{(m,\alpha)} \int_t^T \left(\int_{\T^d} H^*\left(y, \alpha(s,y)\right)m(s,y)dy + {\mathcal F}(m(s))\right)\ ds ,
$$ 
where $(m,\alpha)$ verifies the usual constraint
$$
\partial_t m-\Delta m -\dive(m\alpha)=0\; {\rm in}\; (t,T)\times \T^d, \qquad m(t)=m_0.
$$
Let $(m^T,\alpha^T)$ be a minimizer of the problem, then, $\alpha^T(s,x)=D_p H(x,Du^T(s,x))$, where $(u^T, m^T)$ solves
\begin{equation}\label{psystem}
\begin{cases} 
-\partial_t u-\Delta u+H(x,Du)=F(x,m) & \mbox{in } \T^d\times[t,T]\\
-\partial_t m+\Delta m+{\rm div}(mD_pH(x,Du))=0 & \mbox{in }\T^d\times[t,T]\\
m(t)=m_0,\; u(T,x)=0& \mbox{in }\T^d,
\end{cases}
\end{equation}  
(see for instance \cite{briani2018stable} for details). 

We also take from \cite[Lemma 1.3]{masoero2019} some uniform estimates on the solutions of \eqref{psystem} which will be useful in the next propositions.
\begin{Lemma}\label{2est} There exists $C>0$ independent of $m_0$ and $T$ such that, if $(u,m)$ is a classical solution of \eqref{psystem}, then
	
	\begin{itemize}
		\item $\Vert D u\Vert_{L^\infty([0,T]\times\T^d)}+  \Vert D^2 u\Vert_{L^\infty([0,T]\times\T^d)}\leq C$
		\item $\mathbf d(m(s),m(l))\leq C |l-s|^{1/2}$ for any $l,s\in [0,T]$
	\end{itemize}
	
	Consequently, we also have that $|\partial_{t}u(s,\cdot)|\leq C$ for any $s\in[0,T]$.
\end{Lemma}

\begin{Lemma}\label{lem.Energie} For any $(u,m)$ solution of the MFG system \eqref{psystem}, there exists $c(u,m)\in \R$ such that, for any $t\in [0,T]$,  
	$$
	\int_{\T^d} \Bigl( H(x,Du^T(t,x)) - \Delta u^T(t,x)\Bigr) m^T(t,dx)  - {\mathcal F}(m^T(t)) = c(m,u).
	$$
\end{Lemma}
\begin{proof} As for any $t>0$ both $m^T$ and $u^T$ are smooth in time and space, the integral 
	$$
	\int_{\T^d} H(x,Du(t,x))m^T(t,x) +Du(t,x)\cdot Dm(t,x)dx  - {\mathcal F}(m(t)) 
	$$
	is well defined and we can derive it in time. Then,
	$$
	\int_{\T^d} D_pH(x,Du(t,x))\partial_tDu(t,x)m(t,x)+H(x,Du(t,x))\partial_tm(t,x) +
	$$
	$$
	\int_{\T^d}\partial_tDu(t,x)\cdot Dm(t,x) dx+Du(t,x)\cdot \partial_tDm(t,x) - F(x,m(t))\partial_tm(t,x).
	$$
	Integrating by parts and rearranging the terms we get that the above expression is equal to
	
	$$
	\int_{ \T^d}(-\Delta m(t,x)-\dive(m(t,x)D_p H(x,Du^T(t,x)))\partial_t u^T(t,x)dx+
	$$
	$$
	(-\Delta u^T(t,x)+H(x,Du^T(t,x)-F(x,m^T(t,x)))\partial_tm^T(t,x)dx.
	$$
	If we plug into the last equality the equations verified by $u^T$ and $m^T$, we get
	$$
	\frac{d}{dt}\left(\int_{\T^d} H(x,Du^T(t,x))m^T(t,x) + Du^T(t,x)\cdot Dm^T(t,x) dx - {\mathcal F}(m^T(t))\right)=
	$$
	$$
	\int_{ \T^d}-\partial_t m^T(t,x)\partial_tu^T(t,x)+\partial_t m^T(t,x)\partial_tu^T(t,x)=0.
	$$
	Integrating by parts the term $Du^T(t,x)\cdot Dm^T(t,x)$ and using the continuity of $t\to u(t,\cdot)$ in $C^2(\T^d)$ and the continuity of $t\to m(t)$ in $\Pw$ we conclude that the result holds. 
	
\end{proof}

\begin{Proposition}\label{prop.limcT} Let $(m^T,\alpha^T)$ be optimal for $\mathcal U(T,m_0)$ and $(u^T,m^T)$ be a solution of \eqref{psystem} associated to $(m^T,\alpha^T)$. Then, $c(u^T,m^T)\to \lambda$ as $T\to +\infty$. Moreover, this limit is uniform with respect to the initial condition $m_0$ and the choice of the minimizer $(m^T,\alpha^T)$.
\end{Proposition} 

\begin{proof} We prove it by contradiction. Let us suppose that there exist a sequence $T_i\rightarrow +\infty$ and a sequence $(m^{i},\alpha^i)$, minimizing $\mathcal U^{T_i}(0,m_0^i)$, such that, for any $i\in\N$ and a some $\varepsilon>0$, 
\begin{equation}\label{ccontr}
	\vert c(u^{i},m^{i})-\lambda\vert\geq\varepsilon,
\end{equation}
where, as usual, $\alpha^i(s,x)=D_p H(x,Du^i(s,x))$ and $(u^i, m^i)$ solves \eqref{psystem}. Thanks to Lemma \ref{2est} we know that there exists $C>0$, independent of $i$ such that 
$$
\sup_{T_i>0} \sup_{t \in [0,T_i]} \|\alpha^i(t)\|_\infty+\|D\alpha^i(t)\|_\infty\leq C.
$$
Let $E$ be the set 
$$
E:= \{ \alpha \in W^{1,\infty}(\T^d, \R^d), \; \|\alpha\|_\infty+\|D\alpha\|_\infty\leq C\}.
$$
Then $E$, endowed with the topology of the uniform convergence, is compact. Moreover, $\alpha^i(t)\in E$ for any $t\in [0,T_i]$. Let us introduce the probability measure $\nu^i$ on $\mPT\times E$ by 
$$
\int_{\mPT\times E} f(m,\alpha) \nu^i(dm,d\alpha)= \frac{1}{T_i-1}\int_1^{T_i} f(m^i(t), \alpha^i(t))dt. 
$$
Then $\nu^i$ converges, up to a subsequence denoted in the same way, to some probability measure $\nu$ on $\mPT\times E$. Note that 
\begin{align*}
\frac{1}{T_i}{\mathcal U}(0,m_0^i) & = \frac{1}{T_i} \int_0^1 \left(\int_{\T^d} H^*(y, \alpha^i(s,y))m^i(s,y) dy\right) + {\mathcal F}(m^i(s)))\ ds\\ 
& \qquad + \frac{T_i-1}{T_i} \int_{\mPT\times E} \left(\int_{\T^d} H^*(y, \alpha(y))m(dy) + {\mathcal F}(m)\right)\ \nu^i(dm,d\alpha)
\end{align*}
Hence, as the left-hand side converges, uniformly with respect to $m_0^i$, to $-\lambda$ (see \cite{masoero2019}), we obtain  
\begin{align}\label{kzaberzrsd}
\int_{\mPT\times E} \left(\int_{\T^d} H^*(y, \alpha(y))m(dy) + {\mathcal F}(m)\right)\ \nu(dm,d\alpha)= -\lambda. 
\end{align}
Now we make the link between $\nu$ and the measure $(\mu,p_1)$ of  Section \ref{sec.Mather}. Let $\mu$ be the first marginal of $\nu$ and let us define the vector measure $p_1$ on $\mPT\times \T^d$ as 
$$
\int_{\mPT\times \T^d} \phi(m,y)\cdot p_1(dm, dy) = \int_{\mPT\times E} \int_{\T^d} \phi(m,y)\cdot \alpha(y) m(dy) \nu(dm,d\alpha)
$$
for any test function $\phi\in C^0(\mPT\times \T^d, \R^d)$. We note that $p_1$ is absolutely continuous with respect to $\mu$, since, if we disintegrate $\nu$ with respect to $\mu$: $\nu=\nu_m(d\alpha)\mu(dm)$, then 
$$
p_1(dm, dy)= \int_{E} \alpha(y) m(dy) \nu_m(d\alpha) \mu(dm).
$$
Therefore,
$$
\frac{dp_1}{dm\otimes \mu}(m,y)= \int_{E} \alpha(y)\nu_m(d\alpha) . 
$$
Let us check that $(\mu, p_1)$ is closed. Indeed, for any map $\Phi\in C^{1,1}({\mathcal P}(\T^d))$, we have 
\begin{align*}
\frac{d}{dt} \Phi(m^i(t)) & = \inte  \dive(D_m\Phi(m^i(t),y))m^i(t,dy) - \inte D_m\Phi(m^i(t),y)\cdot H_p(y, Du^i(t,y))m^i(dy) \\
&= \inte  \dive(D_m\Phi(m^i(t),y))m^i(t,dy) + \inte D_m\Phi(m^i(t),y)\cdot \alpha^i(t,y) m^i(dy).
\end{align*}
So,
\begin{align*}
& \int_{\mPT\times E}  \inte \dive(D_m\Phi(m,y)) +  D_m\Phi(m,y)\cdot \alpha(y) m(dy) \ \nu^i(dm,d\alpha)  \\
& \qquad = \frac{1}{T_i-1} \int_1^{T_i} \inte  \dive(D_m\Phi(m^i(t),y))m^i(t,dy) + \inte D_m\Phi(m^i(t),y)\cdot \alpha^i(t,y) m^i(dy) dt \\ 
& \qquad = \frac{1}{T_i-1} \left[\Phi(m^i(T_i))-\Phi(m^i(1))\right].  
\end{align*}
Letting $i\to+\infty$ gives 
\begin{align*}
& \int_{\mPT\times E}  \inte \dive(D_m\Phi(m,y)) +  D_m\Phi(m,y)\cdot \alpha(y) m(dy) \ \nu(dm,d\alpha)=0,
\end{align*}
which can be rewritten, in view of the definition of $p_1$, as 
\begin{align*}
& \int_{\mPT}  \inte \dive(D_m\Phi(m,y)) m(dy) \mu(dm) + \int_{\mPT\times \T^d}  D_m\Phi(m,y)\cdot p_1(dm,dy)=0.
\end{align*}
This proves that $(\mu,p_1)$ is closed. Next we come back to \eqref{kzaberzrsd}: using the convexity of $H^*$, we also have 
\begin{align}
-\lambda=&\int_{\mPT\times E} \left[\int_{\T^d} H^*(y, \alpha(y))m(dy) + {\mathcal F}(m)\right]\ \nu(dm,d\alpha)= \\
&\int_{\mPT}\int_E \left[\int_{\T^d} H^*(y, \alpha(y))m(dy) + {\mathcal F}(m)\right]\ \nu_m(d\alpha)\mu(dm)\geq\\
&\int_{\mPT} \left[\int_{\T^d} H^*\left(y, \int_E\alpha(y)\nu_m(d\alpha)\right)m(dy) + {\mathcal F}(m)\right]\ \mu(dm)=\\
&\int_{\mPT} \left[\int_{\T^d} H^*\left(y, \frac{dp_1}{dm\otimes \mu} (m,y)\right)m(dy) + {\mathcal F}(m)\right]\ \mu(dm).
\end{align}
Therefore,
\begin{equation}\label{12345}
\int_{\mPT} \left[\int_{\T^d} H^*\left(y, \frac{dp_1}{dm\otimes \mu} (m,y)\right)m(dy) + {\mathcal F}(m)\right]\ \mu(dm)\leq -\lambda,
\end{equation}
which proves the minimality of $(\mu, p_1)$. By the uniform convexity of $H$, relation \eqref{12345} shows also that, for $\mu-$a.e. $m\in \mPT$ and for $\nu_m-$a.e. $\alpha$, one has 
\begin{equation}\label{rep.dp1}
\frac{dp_1}{dm\otimes \mu} (m,y) =  \alpha(y).
\end{equation}

Note also that, $m^i(t)$ has a positive density for any $t\in [1,T_i]$ and there exists a constant $C>0$ independent of $i$ such that
\begin{equation}\label{estmt}
\sup_{t\in [1,T_i]} \|1/m^i(t,\cdot)\|_\infty+ \|Dm^i(t,\cdot)\|_\infty\leq C. 
\end{equation}
The bounds on $Dm^i$ are standard and we refer to \cite[Ch4, Theorem 5.1]{ladyzhenskaia1988linear}. While, for the estimates on $1/m^i$, we used the Harnack's inequality in \cite[Theorem 8.1.3]{bogachev2015fokker}. In our setting, this theorem states that, for any $x,y\in\T^d$ and for any $0<s<t<T_i$, there exists a constant $C_{t-s}$, depending only on $\vert t-s\vert$, such that
$$
m^i(t,x)\geq C_{t-s} m^i(s,y).
$$
As we work on the torus, for any $s>0$, there exists a point $y_s\in\T^d$ such that $m^i(s,y_s)\geq1$. Then, we can chose $s=t-1$ and we get that for any $t>1$
$$
m^i(t,x)\geq C_{1} m^i(t-1,y_{t-1})\geq C_1,
$$ 
which proves \eqref{estmt}.

The estimates in \eqref{estmt} ensure that the pair $(\mu,p_1)$ is smooth in the sense of Definition \ref{def.mup1smooth}. In particular, we know by Proposition \ref{eq.lem.mup1} that, for $\mu-$a.e. $m\in \mPT$, 
$$
-\int_{\T^d} q_1(y,m) \cdot Dm(y)dy +   \int_{ \T^d} H(y, q_1(y,m))m(dy) -{\mathcal F}(m)
=\lambda,
$$
where 
$$
q_1(x,m):= D_aH^*\left(y,  \frac{dp_1}{d m\otimes \mu} (y,m)\right). 
$$
By the convergence of $\nu^i$ to $\nu$, there exists (up to a subsequence again) $t_i\in [1,T_i]$ such that $(m^i(t_i),\alpha^i(t_i))$ converges to an element $(m,\alpha)\in \mPT\times E$ which belongs to the support of $\mu$. Then by \eqref{rep.dp1}, $\alpha= \frac{dp_1}{dm\otimes \mu}$. Thus 
$Du^i(t_i)= D_aH^*(y, \alpha^i(t_i))$ converges uniformly to $q_1(\cdot, m)$. This shows that 
\begin{align*}
&\lim_{i\to+\infty} c(u^{i},m^{i})\\
&=\lim_{i\rightarrow+\infty}-\int_{\T^d} Du^{i}(t_i,y)  \cdot Dm^{i}(t_i,y)dy +   \int_{ \T^d} H(y, Du^{i}(t_i,y) )m^{T}(t_i ,dy)- {\mathcal F}(m^{i}(t_i))\\
 & =-\int_{\T^d} q_1(y,m)  \cdot Dm(y)dy +   \int_{ \T^d} H(y, q_1(y,m) )m(dy)- {\mathcal F}(m) = \lambda,
\end{align*}
which is in contradiction with \eqref{ccontr}.
\end{proof}

The next step towards Theorem \ref{teo.xiconv} is to prove that the map $(s,m)\to {\mathcal U}^T(s,m)+\lambda (T-s)$ has a limit. 
In the next proposition we prove that $(s,m)\to {\mathcal U}^T(s,m)+\lambda (T-s)$ is bounded and equicontinuous on $[0,T]\times \mPT$ and so that there exists a subsequence  $({\mathcal U}^{T_n}+\lambda (T_n-\cdot))$ which, locally in time, converges uniformly to a continuous function $\xi$. 

\begin{Proposition}
The maps $(s,m)\to {\mathcal U}^T(s,m)+\lambda (T-s)$ are uniformly bounded and uniformly continuous.
\end{Proposition}
\begin{proof}
	We first prove that $(s,m)\to {\mathcal U}^T(s,m)+\lambda (T-s)$ is bounded, uniformly in $T$. Let $\chi$ be a corrector function. As $\chi$ is a continuous function on the compact set $\mathcal P(\T^{d})$, there exists a constant $C>0$ such that $0\leq \chi(m)+C$ for any $m\in\mathcal P(\T^{d})$. If $(m(t),w(t))$ is an admissible trajectory for the minimization problem of ${\mathcal U}^T(s,m)$, then
	$$
	\int_s^T\int_{\T^d}H^*\left(x,\alpha(t,x)\right)dm(t)+\mathcal F(m(t))dt+\lambda (T-s)
	$$
	$$
	\leq\int_s^T\int_{\T^d}H^*\left(x,\alpha(t,x)\right)dm(t)+\mathcal F(m(t))dt+\chi(m(T))+\lambda (T-s)+C.
	$$
	Taking the infimum over all the possible $(m,\alpha)$, the definition of $\mathcal U^T(s,m)$ and the dynamic programming principle verified by $\chi$ lead to
	$$
	\mathcal U^T(s,m)+\lambda (T-s)\leq \chi(m)+ C.
	$$
	As $\chi$ is bounded, we get an upper bound independent of $T$, $m$ and $s$. The lower bound is analogous.
	
	We turn to the equicontinuity.  For what concern the continuity in the $m$ variable, one can adapt the proof of \cite[Theorem 1.5]{masoero2019} with minor adjustments, to show that, if $T-s\geq \varepsilon>0$ for given $\varepsilon>0$, then there exits a constant $K$ independent of $T$ and $s$ such that $\mathcal U^T(s,\cdot)$ is $K$-Lipschitz continuous.

	 We now need to estimate $\vert\mathcal U^{T}(t_2,m_0)-\mathcal U^T(t_1,m_0)\vert$. We suppose $t_2>t_1$ and we fix $(\bar m(s),\bar \alpha(s))$ an optimal trajectory for $\mathcal U^T(t_1,m_0)$, then
	 \begin{equation}\label{uniftimeUt}
	 \vert\mathcal U^{T}(t_2,m_0)-\mathcal U^T(t_1,m_0)\vert\leq \vert\mathcal U^{T}(t_2,m_0)-\mathcal U^T(t_2,\bar m(t_2))\vert+\vert\mathcal U^{T}(t_2,\bar m(t_2))-\mathcal U^T(t_1,m_0)\vert.
	 \end{equation}
	We can estimate the first term on the right hand-side using at first the uniform Lipschitz continuity we discussed before and then the estimates on the solution of the MFG system in Lemma \ref{2est}. So,
	\begin{equation}\label{halfUtreg}
	\vert\mathcal U^{T}(t_2,m_0)-\mathcal U^T(t_2,\bar m(t_2))\vert\leq K \mathbf{d}(m_0,\bar m(t_2))\leq KC\vert t_1-t_2\vert^{\frac{1}{2}}.
	\end{equation}
	To estimate the second term in the right hand-side of \eqref{uniftimeUt}, we just need to use that
	$$
    {\mathcal U}^T(t_1,m_1)=\inf_{(m,\alpha)}\left\{\int_{t_1}^{t_2}\int_{\T^d}H^*\left(x,\alpha(t,x)\right)dm(t)+{\mathcal F}(m(t))dt+{\mathcal U}^{T}(t_2,m(t_2))\right\}.
	$$
	As $(\bar m,\bar \alpha)$ is optimal for ${\mathcal U}^T(t_1,m_1)$, we get
	$$
	{\mathcal U}^T(t_1,m_1)-{\mathcal U}^{T}(t_2,\bar m(t_2))=\int_{t_1}^{t_2}\int_{\T^d}H^*\left(x,\bar\alpha(t,x)\right)d\bar m(t)+{\mathcal F}(\bar m(t))dt.
	$$
	Note that, $\bar\alpha(t,x)=D_pH(x,D\bar u(x,t))$ where $(\bar m,\bar u)$ solves the MFG system \eqref{psystem} and, according to Lemma \ref{2est}, we have uniform estimates on $\bar u$. Therefore,
	$$
	\vert\mathcal U^{T}(t_2,\bar m(t_2))-\mathcal U^T(t_1,m_0)\vert\leq\int_{t_1}^{t_2}\left\vert\int_{\T^d}H^*\left(x,\bar\alpha(t,x)\right)d\bar m(t)+{\mathcal F}(\bar m(t))\right\vert dt\leq C(t_2-t_1).
	$$
	Putting together the last inequality with \eqref{halfUtreg} we have that, for a possibly different constant $C>0$, independent of $T$, $m_0$ and $m_1$,
	$$
	\vert\mathcal U^{T}(t_2,m_0)-\mathcal U^T(t_1,m_0)\vert\leq C(\vert t_1-t_2\vert^{\frac{1}{2}}+\vert t_1-t_2\vert),
	$$
	which, in turn, implies the uniform continuity in time. 
%
\end{proof}
From now on we fix a continuous map $\xi:[0,+\infty)\times \Pw\to \R$, limit of a subsequence (denoted in the same way) of the sequence $({\mathcal U}^T+\lambda (T-\cdot))$ as $T\to+\infty$.
\begin{Proposition}\label{domination} The map $\xi_0(\cdot):= \xi(0,\cdot)$ satisfies 
	$$
	\xi_0(m)\leq \inf_{(m,\alpha)}\left\{ \int_0^t \left(\int_{\T^d} H^*\left(y, \alpha(s,y)\right)m(s,y)dy + {\mathcal F}(m(s))\right)\ ds +\xi_0(m(t)) \right\}+\lambda t. 
	$$
\end{Proposition}

\begin{proof} We first claim that $\xi$ is a viscosity solution to 
	\begin{equation}\label{eq.noninc}
	-\partial_t \xi \geq 0\qquad {\rm in}\; [0,+\infty)\times \Pw. 
	\end{equation}
	Let $\Phi=\Phi(t,m)$ be a smooth test function such that $\xi\geq \Phi$ with an equality only at $(t_0, m_0)$. Then there exists a subsequence $(t_n, \bar m_n)$ converging to $(t_0,m_0)$ and such that ${\mathcal U}^{T_n}+ \lambda (T_n-s)-\Phi$ has a minimum at $(t_n, \bar m_n)$. Let $(m_n,\alpha_n)$ be a minimizer for ${\mathcal U}^{T_n}(t_n,\bar m_n)$. We consider $u_n\in C^{1,2}([t_n,T_n]\times\T^d)$ such that $(u_n,m_n)$ is a solution to the MFG system \eqref{psystem} and $\alpha_n=D_pH(x,Du_n)$. Then, by Lemma \ref{lem.supsol}, we have 
	$$
	-\partial_t \Phi(t_n,\bar m_n)-\lambda + \inte (H(y, Du_n(t_n,y))-\Delta u_n(t_n,y)) \bar m_n(dy) -{\mathcal F}(\bar m_n)\geq 0. 
	$$
	By Lemma \ref{lem.Energie} and  Proposition \ref{prop.limcT}, we have, given $\ep>0$,  
	$$
	\inte (H(y, D u_n(t_n,y))- \Delta  u_n(t_n,y)) \bar m_n(dy) -{\mathcal F}(\bar m_n)\leq \lambda+\ep,
	$$
	for $n$ large enough. So 
	$$
	-\partial_t \Phi(t_n,\bar  m_n)\geq - \ep.
	$$
	We obtain therefore, after letting $n\to+\infty$ and then $\ep\to 0$,  
	$$
	-\partial_t \Phi(t_0,m_0)\geq 0. 
	$$
	This shows that $\xi$ satisfies \eqref{eq.noninc} holds in the viscosity solution sense. 
	
	We now prove that \eqref{eq.noninc} implies that $\xi$ is nonincreasing in time. Fix $m_0\in \Pw$ and assume on the contrary that there exists $0\leq t_1< t_2$ such that $\xi(t_1,m_0)<\xi(t_2,m_0)$. Let $\Psi=\Phi(m)$ be a smooth test function such that $\Psi> 0$ on $\Pw\backslash \{m_0\}$ with $\Phi(m_0)=0$. Then, we can find $\eta>0$ small such that, if $m_1$ and $m_2$ are such that $\xi(t_1,\cdot)+\eta^{-1}\Psi$ has a minimum at $m_1$ and 
	$\xi(t_2,\cdot)+\eta^{-1}\Psi$ has a minimum at $m_2$, then $\xi(t_1,m_1)<\xi(t_2,m_2)$. Note that this implies that
	\begin{align*}
	& \min_{m\in \Pw} \xi(t_1,m)+\eta^{-1}\Psi(m)= \xi(t_1,m_1)+\eta^{-1}\Psi(m_1) \\
	& \qquad  < \ \xi(t_2,m_2)+\eta^{-1}\Psi(m_2)= \min_{m\in \Pw} \xi(t_2,m)+\eta^{-1}\Psi(m). 
	\end{align*}
	Recalling that $\xi$ is bounded, this implies that we can find $\ep>0$ small such that the map $(t,m)\to \xi(t,m)+\eta^{-1}\Psi(m)+ \ep t$ has an interior minimum on $[t_1, +\infty)\times \Pw$ at some point $(t_3,m_3)\in (t_1,+\infty)\times \Pw$. This contradicts \eqref{eq.noninc}.

	Now that we have proved the monotonicity in time of $\xi$ we can finally show the statement of the proposition. We have that $\mathcal U^T$ verifies the following dynamic programming principle
	\begin{equation}\label{Utdpp}
	{\mathcal U}^T(0,m_0)=\inf_{(m,\alpha)}\left\{\int_0^t\int_{\T^d}H^*\left(x,\alpha(t,x)\right)dm(t)+{\mathcal F}(m(t))dt+{\mathcal U}^{T}(t,m(t))\right\}.
	\end{equation}
	Then, adding on both sides $\lambda T$ and passing to the limit $T\rightarrow+\infty$, one easily checks that $\xi$ satisfies
	$$
	\xi(0,m)= \inf_{(m,\alpha)} \left\{  \int_0^t \left(\int_{\T^d} H^*\left(y, \alpha(s,y)\right)m(s,y)dy + {\mathcal F}(m(s))\right)\ ds +\xi(t,m(t)) \right\} +\lambda t. 
	$$
	Using the fact that $\xi$ is nonincreasing in time we get the desired result. 
\end{proof}

Before we can prove that $\xi_0$ is a corrector function we need to state some standard properties of $\tau_h:C^0(\mathcal P(\T^d))\rightarrow C^0(\mathcal P(\T^d))$ which is defined as follows. For $h>0$ and $\Phi\in C^0\T^d)$ we set 
$$
\tau_h\Phi(m_0)= \inf_{(m,\alpha)} \left\{  \int_0^h \left(\int_{\T^d} H^*\left(y, \alpha(s,y)\right)m(s,y)dy + {\mathcal F}(m(s))\right)\ ds +\Phi(m(h)) \right\} +\lambda h,
$$
where $(m,\alpha)$ solves in the sense of distribution
$$
\begin{cases}
-\partial_t m+\Delta m+\dive (m\alpha)=0&\mbox{ in }[0,h]\times\T^d\\
m(0)=m_0.
\end{cases}
$$
\begin{Lemma}\label{LOprop}
	The function $\tau_h$ verifies the following properties
	\begin{enumerate}
		\item For any $h_1$, $h_2>0$, $\tau_{h_1}\circ\tau_{h_2}=\tau_{h_1+h_2}$ 
		\item For any $h>0$, $\tau_h$ is not expansive, i.e. for any $\Phi$, $\Psi\in C^0\mathcal P(\T^d))$
		$$
		\Vert \tau_h\Phi-\tau_h\Psi\Vert_\infty\leq\Vert \Phi-\Psi\Vert_\infty.
		$$
		\item For any $h>0$, $\tau_h$ is order preserving, i.e. for any $\Phi$, $\Psi\in C^0\mathcal P(\T^d))$ such that $\Phi\leq\Psi$,
		$$
		\tau_h\Phi\leq\tau_h\Psi.
		$$
		\item Let $\Phi\in C^0\mathcal P(\T^d))$ be such that, for any $h>0$, $\Phi\leq\tau_h\Phi$. Then, for any $0<h_1<h_2$,
		$$
		\Phi\leq \tau_{h_1}\Phi\leq\tau_{h_2}\Phi.
		$$
	\end{enumerate}
\begin{proof} The proof is standard, see for instance, in a closely related context, \cite{fathi2008weak}.
\end{proof}
\end{Lemma}
\begin{Theorem}\label{teo.xiconv} $\xi_0$ is a corrector and ${\mathcal U}^T(0)+\lambda T$ converges uniformly to $\xi_0$ on $\mPT$. 
\end{Theorem}

\begin{proof} The proof follows closely the one of \cite[Theorem 6.3.1]{fathi2008weak}. We define $\widetilde{\mathcal U}^T(t,m)={\mathcal U}^{T}(t,m)+\lambda (T-t)$. Let $T_n\rightarrow+\infty$ be a sequence such that $\widetilde{\mathcal U}^{T_n}$ converges locally uniformly to $\xi$ on $[0,+\infty)\times \mPT$. We can suppose that, if we define $s_n=T_{n+1}-T_n$, then $s_n\rightarrow+\infty$. Note that $ \widetilde{\mathcal U}^{T_{n+1}}(s_n,m) = \widetilde{\mathcal U}^{T_n}(0,m)$. Then, using \eqref{Utdpp}, we get
	$$
	\widetilde{\mathcal U}^{T_{n+1}}(0,m)=\tau_{s_n} \widetilde{\mathcal U}^{T_{n+1}}(s_n,m) = \tau_{s_n}\widetilde{\mathcal U}^{T_n}(0,m).
	$$
	 We also know from Lemma \ref{LOprop} that $\tau_h$ is a contraction and that it verifies the semigroup property.
 Therefore, 
\begin{align*}
\| \tau_{s_n}\xi_0-\xi_0\|_\infty &  \leq \| \tau_{s_n} \xi_0-\tau_{s_n} \widetilde{\mathcal U}^{T_n}(0)\|_\infty+ 
\|\tau_{s_n} \widetilde{\mathcal U}^{T_n}(0)- \xi_0\|_\infty\\
& \leq  \| \xi_0-\widetilde{\mathcal U}^{T_n}(0)\|_\infty+ 
\|\widetilde{\mathcal U}^{T_{n+1}}(0)- \xi_0\|_\infty \to 0, 
\end{align*}

Moreover, Proposition \ref{domination} and Lemma \ref{LOprop} prove that $\tau_s\xi_0$ is monotone in $s$. Then, for any $s>0$, we have that, for a sufficiently large $n\in\N$, 
$$
\xi_0\leq\tau_{s}\xi_0\leq\tau_{s_n}\xi_0\rightarrow \xi_0,
$$
 which proves that $\tau_t \xi_0$ is constant in $t$ and so $\xi_0$ is corrector function. It remains to check that the whole sequence $\widetilde{\mathcal U}^T(0)$ converges to $\xi_0$.  Let $T>T_n$, then
 $$
 \Vert \widetilde{\mathcal U}^T(0)-\xi_0\Vert_\infty=\Vert \tau_{T-T_n}\widetilde{\mathcal U}^{T_n}(0)-\tau_{T-T_n}\xi_0\Vert_\infty\leq\Vert \widetilde{\mathcal U}^{T_n}(0)-\xi_0\Vert_\infty\rightarrow0
 $$
 and the result follows. 
 \end{proof}
\subsection{Convergence of optimal trajectories}
Now that we have proved the convergence of $\mathcal U^T(0,\cdot)+\lambda T$ to a corrector function $\chi$, we can properly define the limit trajectories for time dependent MFG and the set where these trajectories lay. We will show that this set is a subset of the projected Mather set $\mathcal M$ as was suggested in \cite{masoero2019}. We recall the definition of $\mathcal M$.

\begin{Definition}\label{def.PMset} We say that $m_0\in \mathcal P(\T^{d})$ belongs to the \textit{projected Mather set} ${\mathcal M}\subset \mathcal P(\T^{d})$ if there exists a calibrated curve $(m(t),\alpha(t))$ such that $m(0)=m_{0}$. 
\end{Definition}
\begin{Remark}\label{rem.Mather}
Note that the notion of projected Mather set that we use here is consistent with the one that was already introduced in \cite{masoero2019}. On the other hand, Definition \ref{def.PMset} is not the transposition of the definition of projected Mather set that is generally used in standard Weak KAM theory. In this latter case the projected Mather set is the union of the projection of the supports of Mather measures on the torus. What we call here projected Mather set would be rather the \textit{projected Aubry set} or the \textit{projected Mané set} (we refer to \cite{fathi2014weak} and \cite{fathi2008weak} for these definitions). We decided to use Definition \ref{def.PMset} mostly to be consistent with the terminology in \cite{masoero2019}. Moreover, it is worthwhile to mention that in the standard theory the Mather set and the Aubry set are deeply connected while in this framework such a relation is no longer clear. In particular, in standard Weak KAM theory the Mather set is contained in the Aubry set (where the latter is defined as the intersection of graphs of calibrated curves). One can check this inclusion defining a calibrated curve, starting from any point of the Mather set, through the Lagrangian flow. In the MFG setting, the lack of uniqueness of solutions and the forward/backward structure of the system prevent from defining any sensible notion of flow. Moreover, an other important difference, that highlights how the connection between Mather set and Aubry set is not clear in the MFG framework, is that, on the one hand, we know that calibrated curves lay on smooth probability measure but, on the other, we know nothing about the regularity of Mather measures' support points (reason why we introduced the notion of "smooth" Mather measure).
\end{Remark}
 We recall also that a couple $(\bar m,\bar\alpha)$, which satisfies $-\partial_t\bar m(t)+\Delta \bar m(t)+\dive(\bar\alpha(t)\bar m(t)=0$ for any $t\in\R$, is a calibrated curve, if there exists a corrector function $\chi:\mathcal P(\T^{d})\rightarrow \R$ such that, for any $t_{1}<t_{2}\in\R$,
\begin{equation}
\chi(\bar m(t_1))=\lambda(t_2-t_1)+ \int_{t_1}^{t_2} \inte H^{*}\left(x, \bar\alpha(s)\right)d\bar m(s) +\mathcal F(\bar m(s))ds+\chi(\bar m(t_2)).
\end{equation}

We fix $(m^T,\alpha^T)$ a minimizer for $\mathcal U^T(-T,m_0)$. As usual $\alpha^T=D_pH(x,D\bar u^T)$ where $(\bar u^T,m^T)$ solves \eqref{psystem} on $[-T,T]\times\T^d$. We define $u^T(t,x)=\bar u^T(t,x)-u^T(0,\bar x)$, for a fixed $\bar x\in\T^d$. We know from Lemma \ref{2est} that $D^2 \bar u^T$ and $\partial_t\bar u^T$ are uniformly bounded. This means that $u^T$ and $Du^T$ are uniformly bounded and uniformly continuous on any compact set of $\R\times\T^d$. Therefore, we have that up to subsequence $u^T$ converges to a function $u\in C^{1,2}(\R\times\T^d)$. The convergence, up to subsequence, of $m^T$ to a function $m\in C^0(\R,\mathcal P(\T^d))$ is ensured again by Lemma \ref{2est} and the uniform $C^{\frac{1}{2}}([0,T],\Pw)$ bounds on $m^T$ therein. It is standard that the couple $(u,m)$ solves in classical sense
$$
\begin{cases}
-\partial_t u-\Delta u+H(x,Du)=F(x,m) & \mbox{in } \R\times\T^{d},\\
-\partial_t m+\Delta m+{\rm div}(mD_pH(x,Du))=0 & \mbox{in }\R\times\T^{d}.
\end{cases}
$$
As by product we have that $(m^T,\alpha^T)$ uniformly converges on compact sets to the couple $(m,D_pH(x,Du))$. Moreover, if we define $\alpha=D_pH(x,Du)$, then $(m,\alpha)$ solves $-\partial_t m+\Delta m+\dive(m\alpha)=0$. We can now prove that $(m,\alpha)$ is a calibrated curve and therefore that $\mathcal M$ contains the uniform limits of optimal trajectories.

\begin{Theorem}\label{teo.malphaconv}
	Let $(m^T,\alpha^T)$ be an optimal trajectory for $\mathcal U^T(-T,m_0)$. Then, $(m^T,\alpha^T)$ converges, up to subsequence, to a calibrated curve $(m,\alpha)$. Consequently, $m(t)\in\mathcal M$ for any $t\in\R$.  
\end{Theorem} 
\begin{proof}
	As we have already discussed the convergence of $(m^T,\alpha^T)$ to $(m,\alpha)$ we just need to check that $(m,\alpha)$ is a calibrated curve. We fix $t_1<t_2\in \R$, then, by dynamic programming principle, 
	$$
	\mathcal U^T(t_1,m_0)+\lambda(T-t_1) =\int_{t_{1}}^{t_{2}} \int_{\T^d}H^*\left(x,\alpha^T(s)\right)dm^T(s)+\mathcal F(m^T(s))ds+\mathcal U^T(t_2,m^T(t_2))+\lambda (T-t_1).
	$$
	We recall that $\mathcal U^T(t,m_0)=\mathcal U^{T-t}(0,m_0)$. 
	Given the continuity of $\mathcal U^T(0,\cdot)$, the uniform convergence of $(m^T,\alpha^T)$ on compact subsets and the uniform convergence of $\mathcal U^T(0,\cdot)+\lambda T$ to $\chi(\cdot)$, we can pass to the limit in $T$ and we get that, for any interval $[t_{1},t_{2}]$, the couple $(m,\alpha)$ verifies	
	$$
	\chi(m(t_{1}))= \int_{t_{1}}^{t_{2}}\inte H^{*}\left(x,\alpha(s)\right)dm(s)+\mathcal F(m(s))ds+\chi (m(t_{2}))+\lambda (t_{2}-t_{1}).
	$$
So 	$(m,\alpha)$ is a calibrated curve. 
\end{proof}

 \section{Appendix}

 \subsection{Proof of Proposition \ref{prop.lacker}}

 \begin{proof}[Proof of Proposition \ref{prop.lacker}] To simplify the notation, we argue as if the convergence of $v^N$ to $V$ in Lemma \ref{lemma.vN.conv.V} holds for the full sequence (i.e., $N_k=N$). We denote with $\mathcal P_2(\R^d)$ the set of Borel probability measures $m$ on $\R^d$ such that 
 	$$
 	\int_{\R^d}\vert x\vert^2m(dx)<+\infty
 	$$
 	 We denote by $\pi:\R^d\to \T^d$ is the usual projection  (and, by abuse of notation, we set $\pi({\bf x})= (\pi(x^1), \dots, \pi(x^N))$ for any ${\bf x}=(x^1, \dots, x^N)\in (\R^d)^N$).
 	 Let $m_0\in\mathcal P_2(\R^d)$ with compact support and $(\xi^i)_{i\in \N}$ be an i.i.d. sequence of random variables on $\R^d$ with law $m_0$. Fix also $(B^i)$ independent Brownian motions on $\R^d$ which are also independent of $(\xi^i)$ and set ${\bf \xi}^N=(\xi^1, \dots, \xi^N)$. When we look at $v^N$ as a $\Z^d-$periodic map on $(\R^d)^N$, we have by classical representation formula that
 	\begin{equation}\label{nplayerpb}
 	\E\left[ v^N({\bf \xi}^N)\right] = \inf_{{\bf \alpha}} \E\left[ \sum_{i=1}^N \int_0^T H^*(X^i_t,\alpha^i_t) dt + {\mathcal F} (\pi\sharp m^N_{{\bf X}_t})\ dt + v^N({\bf X}_T)\right]+\lambda^N T, 
 	\end{equation}
 	where the infimum is taken over progressively measurable controls ${\bf \alpha}= (\alpha^1, \dots, \alpha^N)$ (adapted to the filtration generated by the $(B^i)$ and the $(\xi^i)$) and where ${\bf X}= (X^1, \dots, X^N)$ solves 
 	$$
 	dX^i_t= \alpha^i_tdt + \sqrt{2} dB^i_t,\; t\in [0,T], \qquad X^i_0=\xi^i.
 	$$
 	The optimal feedback in \eqref{nplayerpb} is well-known: it is given by $\alpha^{*i,N}(t,{\bf x}):=D_pH(x_i,D_{x_i}v^N({\bf x}))$. We denote by ${\bf \alpha}^{*N}=(\alpha^{*1,N}, \dots, \alpha^{*N,N})$ and ${\bf X}^{*N}=(X^{*1,N}, \dots, X^{*N,N})$ the corresponding optimal solution: 
 	$$
 	dX^{*i,N}_t= \alpha^{*i,N}_tdt + \sqrt{2} dB^i_t,\; t\in [0,T], \qquad X^{*i,N}_0=\xi^i.
 	$$
 	
 	By Lemma \ref{lemma.vN.conv.V}, we have 
 	\begin{equation} \label{zaeimsdcj}
 	\lim_{N\to +\infty}  \ep_N=0 \qquad {\rm where}\qquad \ep_N:=  \sup_{{\bf x}\in (\R^d)^N} \left| V(\pi\sharp m^N_{{\bf x}})- v^N({\bf x})\right|, 
 	\end{equation}
 	because $v^N({\bf x})= v^N((\pi(x^1), \dots, \pi(x^N))$ while $\pi\sharp m^N_{{\bf x}}= m^N_{\pi({\bf x})}$. 
 	
 	Note that \eqref{zaeimsdcj} implies, on the one hand, that 
 	\begin{align}\label{lkjnsclklkm}
 	\left| \E\left[v^N({\bf \xi^N})\right] - \E\left[ \sum_{i=1}^N \int_0^T H^*(X^{*i,N}_t,\alpha^{*i,N}_t) dt + {\mathcal F} (\pi\sharp m^N_{{\bf X}^{*N}_t})\ dt + V(\pi\sharp m^N_{{\bf X}^{*N}_T}) +\lambda^NT\right]\right|  \leq \ep_N,
 	\end{align}
 	and, on the other hand, that ${\bf \alpha}^{*N}$ is $2\ep_N$ optimal for the problem if one replaces $v^N$ by $V$ in the optimal control problem: 
 	\begin{align}\label{ljkanzesx}
 	& \Bigl| \E\left[ \sum_{i=1}^N \int_0^T H^*(X^{*i,N}_t,\alpha^{*i,N}_t) dt + {\mathcal F} (\pi\sharp m^N_{{\bf X}^{*N}_t})\ dt + V(\pi\sharp m^N_{{\bf X}^{*N}_T})\right]\notag  \\
 	& \qquad - \inf_{{\bf \alpha}} \E\left[ \sum_{i=1}^N \int_0^T H^*(X^i_t,\alpha^i_t) dt + {\mathcal F} (\pi\sharp m^N_{{\bf X}_t})\ dt + V(\pi\sharp m^N_{{\bf X}_T})\right]\Bigr| \leq  2 \ep_N.
 	\end{align}
 	
 	We aim at letting $N\to +\infty$ in the above inequalities. By the law of large numbers we know that $(m^N_{{\bf \xi}^N})$ converges a.s. and in expectation in ${\mathcal P}_2$ to $m_0$. Therefore, by the Lipschitz continuity of $V$, we obtain  
 	\begin{align}\label{lkjnsclklkmBIS}
 	& \limsup  \left| \E\left[v^N({\bf \xi}^N)\right] - V(\pi\sharp m_0)\right|  \notag \\
 	& \qquad   \leq\limsup \E\left[ \left| v^N({\bf \xi}^N)- V(\pi\sharp m^N_{{\bf \xi}^N})\right| \right] 
 	+ \limsup \E\left[ \left|  V(\pi\sharp m^N_{{\bf \xi}^N})- V(\pi\sharp m_0) \right| \right]  \\
 	&\qquad  \leq \limsup  \ep_N+ C  \E\left[ {\bf d}_2 (m^N_{{\bf \xi}^N}, m_0)\right]  = 0.   \notag
 	\end{align}
 	In order to pass to the limit in \eqref{ljkanzesx}, we use several results of \cite{lacker2017limit}. The first one  (Corollary 2.13) states that $m^N_{{\bf X}^{*N}}$ is precompact in ${\mathcal P}_2([0,T], {\mathcal P}_2(\R^d))$ and that every weak limit has a support in the set of relaxed minimizers of the McKean Vlasov optimal control problem  in weak Markovian formulation (expressed here with---almost---the notation of \cite{lacker2017limit}, see Proposition 2.5): 
 	\begin{equation}\label{relaxedpb}
 	\kappa:= \inf_{\P} \E^\P\left[ \int_0^T H^*(X_t, \hat \alpha(t,X_t)) + {\mathcal F}(\pi\sharp ( \P\circ X_t^{-1})) \ dt + V(\pi\sharp (\P\circ X_T^{-1})) \right]+\lambda^*T,
 	\end{equation}
 	where the infimum is taken over the family of probability spaces $(\Omega, ({\mathcal F}_t), \P)$  supporting a $d-$dimensional process $X$ and a $d-$dimensional Brownian motion $B$, such that $\P\circ X_0^{-1}= m_0$, $\hat \alpha: [0,T]\times \R^d \to \R^d$ is Borel measurable, and the following  holds: 
 	\begin{equation}\label{lkejsdn}
 	dX_t= \hat \alpha(t,X_t)dt +\sqrt{2}dB_t, 
 	\qquad 
 	{\rm with} 
 	\qquad
 	\E^\P\left[ \int_0^T |X_t|^2 + |\hat \alpha(t,X_t)|^2 \ dt \right] <+\infty.
 	\end{equation}
 	The above problem can be reformulated in PDE term as  
 	\begin{equation}\label{aemsrdlcgv}
 	\kappa= \inf_{(m,\hat \alpha)} \int_0^T \int_{\R^d} H^*(x,\hat \alpha(t,x)) m(t,dx) dt + \int_0^T {\mathcal F}(\pi\sharp m(t))dt + V(\pi\sharp m(T))+\lambda^*T,
 	\end{equation}
 	where the infimum is taken over the pairs $(m,\hat \alpha)$ with $m\in C^0([0,T], {\mathcal P}_2(\R^d))$, $\hat \alpha\in L^2([0,T], L^2_{m(t)}(\R^d))$ and 
 	\begin{equation}\label{sfdkgfcontcont}
 	\partial_t m -\Delta m -{\rm div}(m\hat \alpha)=0 \quad {\rm in}\; (0,T)\times \R^d, \qquad m(0)=m_0.
 	\end{equation}
 	Indeed, if $\P$ is admissible in \eqref{relaxedpb},  we just need to set $m(t)= \P\circ X_t^{-1}$ and then $(m,\hat \alpha)$ is admissible in \eqref{aemsrdlcgv}. Conversely, if $(m,\hat \alpha)$ is admissible in \eqref{aemsrdlcgv}, then the exists a weak solution to the SDE  \eqref{lkejsdn}: this  precisely means that there exists a stochastic basis which is admissible for \eqref{relaxedpb}.  
 	
 	Next, we note that the proof of Theorem 2.11 in \cite{lacker2017limit} (and more precisely inequality (6.1))  shows that 
 	$$
 	\lim_N \E\left[ \sum_{i=1}^N \int_0^T H^*(X^{*i,N}_t,\alpha^{*i,N}_t) dt + {\mathcal F} (\pi\sharp m^N_{{\bf X}^{*N}_t})\ dt + V(\pi\sharp m^N_{{\bf X}^{*N}_T})\right] +\lambda^*T= \kappa . 
 	$$
 	Putting together \eqref{lkjnsclklkm}, \eqref{lkjnsclklkmBIS}, \eqref{aemsrdlcgv} and the above equality shows that 
 	\begin{align}\label{hqebsdrejnqdn}
 	& V(\pi\sharp m_0)  = 
 	\lim_N \E\left[ v^N({\bf \xi}^N) \right]  \notag\\
 	&  = 
 	\lim_N \E\left[ \sum_{i=1}^N \int_0^T H^*(X^{*i,N}_t,\alpha^{*i,N}_t) dt + {\mathcal F} (\pi\sharp m^N_{{\bf X}^{*N}_t})\ dt + V(\pi\sharp m^N_{{\bf X}^{*N}_T})\right]+\lambda^*T   \\
 	&   = \kappa=  \inf_{(m,\hat \alpha)} \int_0^T \int_{\R^d} H^*(x,\hat \alpha(t,x)) m(t,dx) dt + \int_0^T {\mathcal F}(\pi\sharp m(t))dt + V(\pi\sharp m(T))+\lambda^*T,\notag  
 	\end{align}
 	where the infimum is computed as above. 
 	
 	It remains to explain why we can work in $\T^d$ instead of $\R^d$. For this, let us  define, for any $m_0\in\mathcal P(\R^d)$ and any $\tilde m_0\in\mathcal P_2(\T^d)$, $J_{\R^d}(m_0)$  and $J_{\T^d}(\tilde m_0)$ by 
 	\begin{align*}
 	&J_{\R^d}(m_0):=  \inf_{(m,\hat \alpha)} \int_0^T \int_{\R^d} H^*(x,\hat \alpha(t,x)) m(t,dx) dt + \int_0^T {\mathcal F}(\pi\sharp m(t))dt + V(\pi\sharp m(T)), 
 	\end{align*} 
where the infimum is taken over the pairs $(m,\hat \alpha)$ as above,
	and 
 	\begin{align*}
 	& J_{\T^d}(\tilde m_0):=  \inf_{(\tilde m, \tilde w)} \int_0^T \int_{\T^d} H^*(x, -\frac{d\tilde w}{d\tilde m}(t,x))\tilde m(t,dx)dt+ \int_0^T {\mathcal F}(\tilde m(t))dt + V(\tilde m(T)), 
 	\end{align*}
 	where the infimum is computed (as usual) over the pairs $(\tilde m,w)$ such that $\tilde m\in C^0([0,T], \mathcal P(\T^d))$, $\tilde w$ is a vector measure on $[0,T]\times \T^d$ with values in $\R^d$ with first marginal $dt$ and which is absolutely continuous with respect to $\tilde m$ with
 	$$
 	\int_0^T \int_{\T^d} \Bigl| \frac{d\tilde w}{d\tilde m}(t,x)\Bigr|^2 \tilde m(t,dx)dt <+\infty
 	$$
 	and the continuity equation 
 	$$
 	\partial_t  \tilde m -\Delta  \tilde m +{\rm div}(\tilde  w)=0 \; {\rm in}\; (0,T)\times \T^d, \qquad \tilde  m(0)=\tilde m_0
 	$$
 	holds.	
From the Lipschitz continuity of $V$ and the convexity of $H^*$ one can easily prove that $J_{\R^d}$ is continuous on $\mathcal P_2(\R^d)$. Our aim is to show that $J_{\R^d}(m_0)= J_{\T^d}(\pi\sharp m_0)$ and $J_{\T^d}(\tilde m_0)= J_{\R^d}(\tilde m_0{\bf 1}_{Q_1})$ for any $m_0\in\mathcal P_2(\R^d)$ and any $\tilde m_0\in\mathcal P(\T^d)$, where $Q_1=[-1/2,1/2)^d$.

	Let $m_0\in {\mathcal P}_2(\R^d)$ and $(m,\hat \alpha)$ be admissible for $J_{\R^d}(m_0)$. We see $w:= -\hat \alpha m$ as a vector measure on $[0,T]\times \R^d$ with values in $\R^d$ and finite mass (since $\hat \alpha \in L^2([0,T], L^2_{m(t)}(\R^d))$).  Let us set $\tilde m(t)=\pi\sharp m(t)$ and $\tilde w(t):= \pi \sharp w(t)$. Then, as $(m,\alpha)$ satisfies \eqref{sfdkgfcontcont},  $(\tilde m, \tilde w)$ solves the continuity equation 
 	$$
 	\partial_t \tilde m -\Delta \tilde m +{\rm div}(\tilde w)=0 \; {\rm in}\; (0,T)\times \T^d, \qquad \tilde m(0)=\pi\sharp m_0. 
 	$$	
 	Indeed, if $\varphi$ is a smooth function with compact support on $[0,T)\times\T^d$ we have
 	\begin{align*}
 	&\int_0^T\int_{\T^d}(\partial_t\varphi(t,x)+\Delta \varphi(t,x))\tilde m(t,dx)+D\varphi(t,x)\tilde w(t,dx) + \int_{\T^d} \phi(0)m_0\\
 	& \qquad = \int_0^T\int_{\R^d}(\partial_t\varphi(t,\pi(y))+\Delta \varphi(t,\pi(y))) m(t,dy)+D\varphi(t,\pi(y)) w(t,dy) + \int_{\T^d} \phi(0)m_0 =0,
 	\end{align*}
 	where the last equality holds because the map $(t,y)\mapsto \varphi(t,\pi(y))$ is smooth and bounded with bounded derivatives on $[0,T)\times\R^d$ and $(m,w)$ verifies $-\partial_t m+\Delta m-\dive w=0$. As
 	\begin{align*}
 	\int_0^T \int_{\T^d} H^*(x, \frac{d\tilde w}{d\tilde m}(t,x))\tilde m(t,dx)dt & = \int_0^T\int_{\R^d} H^*(\pi(x), \frac{dw}{dm}(t,x))m(t,dx)dt  \\
 	& = \int_0^T \int_{\R^d} H^*(x, \hat \alpha(t,x))m(t,dx)dt,
 	\end{align*}
 	(since $L$ is $\Z^d$ periodic in the first variable), we easily derive that  $J_{\R^d}(m_0)\geq J_{\T^d}(\pi\sharp m_0)$.

 To prove the opposite inequality, let now $(\tilde m,\tilde w)$ be $\theta$-optimal for $J_{\T^d}(\tilde m_0)$. We define $(\tilde m_\varepsilon,\tilde w_\varepsilon):=(\tilde m*\xi_{\varepsilon},\tilde w*\xi_{\varepsilon})$ where $\xi_\varepsilon$ is a standard mollification kernel in space. The couple $(\tilde m_\varepsilon,\tilde w_\varepsilon)$ solves the usual continuity equation with initial condition $\tilde m_{0,\varepsilon}=\tilde m_0*\xi_\varepsilon$. Then
 		\begin{align}\label{}
 		\int_0^T \int_{\T^d} H^*(x, -\frac{d\tilde w_\varepsilon}{d\tilde m_\varepsilon}(t,x))\tilde m_\varepsilon(t,dx)dt+ \int_0^T {\mathcal F}(\tilde m_\varepsilon(t))dt + V(\tilde m_\varepsilon(T)) \leq J_{\T^d}(\tilde m_0)+ \theta+ O_\varepsilon(1).
 		\end{align} 		
 	Set $\tilde \alpha_\varepsilon(t,x)= -(\tilde w_\varepsilon/\tilde m_\varepsilon)(t,x)$. Let $m_\varepsilon$ be the solution to 
 	$$
 	\partial_t m_\varepsilon -\Delta m_\varepsilon - {\rm div}(\tilde  \alpha_\varepsilon(t,\pi(x))m_\varepsilon)=0 \; {\rm in}\; (0,T)\times \R^d, \qquad  m_\varepsilon(0,y)= \tilde m_{0,\varepsilon}(\pi(y)){\bf 1}_{Q_1}(y) \quad y\in\R^d. 
 	$$
 	Then, by periodicity of $\tilde\alpha_\varepsilon$, $ \tilde \mu_\varepsilon(t):= \pi\sharp m_\varepsilon(t)$ solves 
 	$$
 	\partial_t \tilde \mu_\varepsilon -\Delta \tilde \mu_\varepsilon - {\rm div}(\tilde  \alpha_\varepsilon(t,x)\tilde \mu_\varepsilon )=0 \; {\rm in}\; (0,T)\times \T^d, \qquad \tilde  \mu_\varepsilon(0)=\pi\sharp m_\varepsilon(0)= \tilde m_{0,\varepsilon}, 
 	$$
 	which has $\tilde m_\varepsilon$ has unique solution since $\tilde  \alpha_\varepsilon$ is smooth in space. This shows that $\pi\sharp m_\varepsilon(t)= \tilde m_\varepsilon(t)$.  Therefore, as $( m_\varepsilon(\cdot,\cdot),\tilde\alpha_\varepsilon(\cdot,\pi(\cdot)))$ is an admissible competitor for $J_{\R^d}(m_\varepsilon(0))$, we get 
 	\begin{align*}
 	&J_{\R^d}(m_\varepsilon(0))\leq \int_0^T \int_{\R^d} H^*(x,  \tilde \alpha_\varepsilon(t,\pi(x))) m_\varepsilon(t,dx)dt+ \int_0^T {\mathcal F}(\pi\sharp  m_\varepsilon(t))dt +V(\pi \sharp  m_\varepsilon(T)) \\
 	& \qquad = \int_0^T \int_{\T^d} H^*(x, \tilde \alpha_\varepsilon(t,x))\tilde m_\varepsilon(t,dx)dt + \int_0^T {\mathcal F}(\tilde m_\varepsilon(t))dt +V(\tilde m_\varepsilon(T)) \leq J_{\T^d}(\tilde m_0)+ \theta+ O_\varepsilon(1).
 	\end{align*}
 	
 	As $J_{\R^d}$ is continuous, we can pass to the limit as $\ep$ and then $\theta$ tend to $0$. Then we find that $J_{\R^d}(\tilde m_0{\bf 1}_{Q_1})\leq J_{\T^d}(\tilde m_0)$ and, therefore, that $J_{\R^d}(\tilde m_0{\bf 1}_{Q_1})= J_{\T^d}(\tilde m_0)$. In view of \eqref{hqebsdrejnqdn}, this completes the proof of the proposition.
 \end{proof}

  \subsection{A viscosity solution property} 
 
 \begin{Lemma}\label{lem.supsol} Let $\Phi= \Phi(t,m)$ be a smooth test function such that ${\mathcal U}^T-\Phi$ has a minimum at a point $(t_0,m_0)\in [0,T)\times \Pw$. Let $(u,m)$ be a solution of the MFG system \eqref{psystem} starting from $(t_0,m_0)$ and such that $(m, D_pH(x,Du(t,x))$ is optimal for ${\mathcal U}(t_0,m_0)$. Then 
 	\begin{equation}\label{DmPhiDu}
 	D_m\Phi(t_0, m_0,x)= Du(t_0,x) \;\mbox{\rm for $m_0-$a.e. $x\in \T^d$}
 	\end{equation}
 	and 
 	\begin{equation}\label{intequ}
 	-\partial_t\Phi(t_0,m_0)+\inte (H(x, Du(t_0,x)-\Delta u(t_0,x))m_0(dx) -{\mathcal F}(m_0) \geq 0.
 	\end{equation}
 \end{Lemma} 
 
 \begin{proof} Without loss of generality we assume that $\Phi(t_0,m_0)={\mathcal U}^T(t_0,m_0)$.   
 	Let $m'(t_0)\in \Pw$ and $m'=m'(t)$ be the solution to 
 	$$
 	\partial_t m'-\Delta m'-{\rm div}(m'D_pH(x,Du(t,x))=0.
 	$$
 	We set $\mu(t)=m'(t)-m(t)$ and, for $h\in (0,1]$ and note that the pair $(m+h\mu, D_pH(x,Du(t,x))$ is a solution to $-\partial_t m+\Delta m+\dive (m\alpha)=0$ with initial condition $m^h_0:=(1-h)m_0+hm'(t_0)$. Hence, by the definition of $\Phi$ and ${\mathcal U}^T$ we have 
 	\begin{align*}
 	&\Phi(t_0,m^h_0)\leq {\mathcal U}^T(t_0,m^h_0)\\
 	&\qquad  \leq \int_{t_0}^T (\inte H^*(x,D_pH(x,Du(t,x)))(m+h\mu)(t,dx)+{\mathcal F}((m+h\mu)(t))dt\\
 	&\qquad\leq\Phi(t_0, m_0)+h\Bigl(\int_{t_0}^T \inte(H^*(x,D_pH(x,Du(t,x)))+F(x,m(t)))\mu(t,dx)dt+o_h(1)\Bigr).
 	\end{align*}
 	Next we use the equation for $u$ and then for $\mu$: 
 	\begin{align*}
 	& \int_{t_0}^T \inte(H^*(x,D_pH(x,Du(t,x)))+F(x,m(t)))\mu(t,dx)dt\\ 
 	& = \int_{t_0}^T \inte(H^*(x,D_pH(x,Du(t,x)))-\partial_t u-\Delta u +H(x,Du(t,x)))\mu(t,dx)dt\\
 	& = \inte u(t_0,x)\mu(t_0,dx)= \inte u(t_0,x)(m'(t_0)-m_0)(dx).
 	\end{align*}
 	Plugging this into the estimate of  $\Phi(t_0,m^h_0)$ above, we obtain, dividing by $h$ and letting $h\to0$,  
 	\begin{align*}
 	\inte \frac{\delta\Phi(t_0,m_0,x)}{\delta m}(m'(t_0)-m_0)(dx) \leq \inte u(t_0,x)(m'(t_0)-m_0)(dx). 
 	\end{align*}	
	Recalling the convention \eqref{deriv.conve} on the derivative and the arbitrariness of $m'(t_0)$, we infer, by choosing Dirac masses for $m'(t_0)$, that 
 	$$
 	\frac{\delta\Phi(t_0,m_0,x)}{\delta m}\leq u(t_0,x) - \inte u(t_0,y)m_0(dy) \qquad \forall x\in \T^d, 
 	$$
 	while 
 	$$
 	\inte \frac{\delta\Phi(t_0,m_0,x)}{\delta m}m_0(dx) = \inte \left(u(t_0,x) - \inte u(t_0,y)m_0(dy)\right)m_0(dx) \; =0. 
 	$$
 	Therefore
 	$$
	\frac{\delta\Phi(t_0,m_0,x)}{\delta m}= u(t_0,x) - \inte u(t_0,y)m_0(dy) , \qquad \mbox{\rm $m_0-$a.e. $x\in \T^d$.}
 	$$
 	This shows that the map $x\to \frac{\delta\Phi(t_0,m_0,x)}{\delta m}-u(t_0,x)$ has a maximum on $\T^d$ at $m_0-$a.e. $x\in \T^d$ and thus \eqref{DmPhiDu} holds. 
 	
 	As ${\mathcal U}^T$ satisfies a dynamic programming principle and ${\mathcal U}^T-\Phi$ has a minimum at $(t_0,m_0)$, it is standard that $\Phi$ also satisfies 
 	$$
 	-\partial_t\Phi(t_0,m_0)+\inte H(x, D_m\Phi(t_0,m_0,x))m_0(dx)-\inte {\rm div}_yD_m\Phi(t_0,m_0,x)m_0(dx) -{\mathcal F}(m_0) \geq 0.
 	$$
 	Using \eqref{DmPhiDu} one then infers that \eqref{intequ} holds. 
 \end{proof}
 \subsection{Smooth test functions}
 Here we fix a corrector $\chi$ and construct a smooth function that touches $\chi$ from above. We fix $m_{0}\in\mathcal P(\T^{d})$ and $\tau>0$. We know from \cite[Appendix]{masoero2019} that, if $(\bar m,\bar\alpha)$ verifies
 
 \begin{equation}\label{nchie}
 \chi(m_0)=\int_0^{2\tau}\int_{\T^d}H^*(x,\bar \alpha)d\bar m(s)+\mathcal F(\bar m(s))ds+2\lambda \tau +\chi(\bar m_{2\tau}),
 \end{equation}
 then there exists a couple $(\bar u,\bar m)$ which solves the MFG system
 \begin{equation}\label{nchiesys}
 \begin{cases} 
 -\partial_t u-\Delta u+H(x,Du)=F(x,m) & \mbox{in } \T^d\times[0,2\tau],\\
 -\partial_t m+\Delta m+{\rm div}(mD_pH(x,Du))=0 & \mbox{in }\T^d\times[0,2\tau],\\
 m(0)=m_0,
 \end{cases}
 \end{equation}
 such that $\bar \alpha=D_{p}H(x,D\bar u)$.
 
 For any $m_{1}\in\mathcal P(\T^{d})$ we define $m(t)$ and $\alpha(t)$ as follows. We first consider $\tilde m$ solution of 
 \begin{equation}\label{FP}
 \begin{cases}
 -\partial_t \tilde m +\Delta \tilde m+{\rm div}(\tilde mD_pH(x,D\bar u))=0 & \mbox{in }\,[0,\tau]\times\T^d\\
 \tilde m(0)=m_{1}
 \end{cases}
 \end{equation}
 and then we set
 \begin{equation}
 m(s,x)=\begin{cases} 
 	\tilde m(s,x), & \mbox{if }s\in(0,\tau] \\ 
 	\frac{2\tau-s}{\tau}\tilde m(s,x)+\frac{s-\tau}{\tau}\bar m(s,x), & \mbox{if }s\in[\tau,2\tau].
 \end{cases}
 \end{equation}
 
 Let $\zeta:[0,2\tau]\times\T^d\rightarrow\R$ be the solution to $\Delta\zeta(t)=\bar m(t)-\tilde m(t)$ so that $\int_{\T^d}\zeta(s,x)=0$. Then, the drift $\alpha$ will be $\alpha(t,x)=D_pH(x,D\bar u(t,x))+\frac{D\zeta(t,x)}{\tau\,m(t,x)}$ in $[\tau,2\tau]$ and $\alpha(t,x)=D_pH(x,D\bar u(t,x))$ elsewhere.
 
 We define the function $\Psi(m_{1})$ as
 
 \begin{equation}\label{Psi}
 \Psi(m_{1})=\int_{0}^{2\tau}\int_{\T^{d}}H^{*}(x,\alpha(t))dm(t)+\mathcal F(m(t))dt+2\lambda\tau+\chi(\bar m(2\tau)).
 \end{equation}
 
 \begin{Proposition}\label{Psiconstr}
 The function $\Psi:\mathcal{P}(\T^d)\rightarrow\R$ defined in \eqref{Psi} is twice differentiable with respect to $m$ with $C^2$ continuous derivatives in space and with derivatives bounded independently of $\chi$.
 \end{Proposition}
 \begin{proof}
 	We first introduce $\Gamma:\R^+\times\T^d\times\R^+\times\T^d\rightarrow\R$, the fundamental solution of \eqref{FP}, i.e. $\Gamma(\cdot,\cdot\,;s,x)$ is the solution of \eqref{FP} starting at time $s$ with initial condition $\Gamma(s,y;s,x)=\delta_x(y)$.
 	Then, by superposition, the solution $\bar m(t)$ of \eqref{FP} is given by $\bar m(t,x)= \int_{\T^d} \Gamma(t,x; 0, y)m_1(dy)$ (for $t>0$).
 	
 	We consider separately the following two integrals:
 	\begin{equation}\label{I1}
 	I_1(m_1)=\int_{0}^{\tau}\int_{\T^{d}}H^{*}(x,\bar\alpha(t))d\bar m(t)+\mathcal F(\bar m(t))dt
 	\end{equation}
 	and
 	\begin{equation}
 	I_2(m_1)=\int_{\tau}^{2\tau}\int_{\T^{d}}H^{*}(x,\alpha(t))dm(t)+\mathcal F(m(t))dt.
 	\end{equation}
 	Note that $\Psi=I_1+I_2+2\lambda\tau+\chi(\bar m(2\tau))$.
 	
 	If we plug $\Gamma$ into \eqref{I1}, then 
 	\begin{equation}
 	I_1(m_1)=\int_{0}^{\tau}\int_{\T^{d}}\int_{\T^{d}}H^{*}(x,\bar\alpha(t))\Gamma(t,x,0,y) m_1(dy)dxdt+\mathcal F\left(\int_{\T^{d}}\Gamma(t,\cdot,0,y) m_1(dy)\right)dt
 	\end{equation}
 	We can now derive $I_1$ with respect to $m_1$ andwe get
 	\begin{equation}
 	\frac{\delta I_1}{\delta m}(m_1,y)=\int_{0}^{\tau}\int_{\T^{d}}H^{*}(x,\bar\alpha(t))\Gamma(t,x,0,y)dxdt+\int_0^\tau\int_{\T^{d}}F(x,\bar m(t))\Gamma(t,x,0,y)dxdt.
 	\end{equation}
 	As the functions $\bar\alpha$ and $\bar m$ are smooth, standar results in parabolic equation ensures that $\delta I_1/\delta m(m_1,\cdot)$ is smooth (see for instance Chapter $4$§$14$ in \cite{ladyzhenskaia1988linear}).
 	
 	We now focus on $I_2$. We fix $G:\T^d\rightarrow\R$ the kernel associated to the integral representation of the solution of the Poisson equation (see for instance \cite[Theorem 4.13]{aubin2013some}). Explicitly, if $\Delta\zeta=f$, then 
 	\begin{equation}\label{kp}
 	\zeta(x)=\int_{\T^d}G(x-y)f(y)dy
 	\end{equation}

 	We first analyse the integral $\int_{\tau}^{2\tau}\int_{\T^{d}}H^{*}(x,\alpha(t))dm(t)$, which explicitly becomes
 	\begin{equation*}
 	\int_{\tau}^{2\tau}\int_{\T^{d}}H^{*}\left(x,D_pH(x,D\bar u)+\frac{D\zeta}{(2\tau-s)\tilde m+(s-\tau)\bar m}\right)\left(\frac{2\tau-s}{\tau}\tilde m(s,x)+\frac{s-\tau}{\tau}\bar m(s,x)\right).
 	\end{equation*}
 	
 	If we derive the above expresion we get
 	\begin{equation}\label{fder}
 	\begin{aligned}
 	\frac{1}{\tau}\int_{\tau}^{2\tau}\int_{\T^{d}}\int_{\T^{d}}-D_pH^*(\alpha)\frac{G(x-z)D_z\Gamma_0(t,z;y)m(t,x)+(2\tau-s)\Gamma_0(t,x;y)D\zeta(t,x)}{m(t,x)}dxdz
 	\\
 	+\frac{1}{\tau}\int_{\tau}^{2\tau}\int_{\T^{d}}\int_{\T^{d}}H^*(\alpha)(2\tau-s)\Gamma_0(t,x;y)dxdz.
 	\end{aligned}
 	\end{equation}
 	Therefore,$\frac{\delta I_2}{\delta m}(m_1,y)$ is equal to
 	\begin{align*}
 	&\frac{1}{\tau}\int_{\tau}^{2\tau}\int_{\T^{d}}\int_{\T^{d}}-D_pH^*(\alpha)\frac{G(x-z)D_z\Gamma_0(t,z;y)m(t,x)+(2\tau-s)\Gamma_0(t,x;y)D\zeta(t,x)}{m(t,x)}dxdzdt+
 	\\
 	&\frac{1}{\tau}\int_{\tau}^{2\tau}\int_{\T^{d}}\int_{\T^{d}}H^*(x,\alpha)(2\tau-s)\Gamma_0(t,x;y)dxdz+\frac{1}{\tau}\int_{\tau}^{2\tau}\int_{\T^{d}}(2\tau-s)F(x,m(s,x))\Gamma_0(t,x;y)dxdt
 	\end{align*}
 	and 
 	\begin{equation}
 	\frac{\delta \Psi}{\delta m}(m_1,y)=\frac{\delta I_1}{\delta m}(m_1,y)+\frac{\delta I_2}{\delta m}(m_1,y)
 	\end{equation}
 	
 	Note that, as in the above expression we are looking at a time interval bounded away from zero, the parabolic regularity ensures that all the functions therein are smooth with respect to the state variable. This implies that also $D_m\Psi$ is well defined.
 	
 	We omit the proof for second order derivatives. It does not present any further difficulties. Indeed, the parabolic regularity, enjoyed by the solutions of the MFG system at any time $t>0$, ensures that we can deploy the same kind of computations that we used in \eqref{fder} and so that both $D^2_{mm}\Psi$ and $D^2_{my}\Psi$ are well defined and bounded.
 \end{proof}
 \begin{Lemma}\label{chitestf}
For any $m_0\in\mathcal P(\T^d)$, there exists a function $\bar\Psi\in C^2(\mathcal P(\T^d))$ such that $\bar\Psi(m)>\chi(m)$ for any $m\neq m_0$ and $\bar\Psi(m_0)=\chi(m_0)$. 

Moreover, we can choose $\bar\Psi$ such that $D^2_{mm}\bar\Psi$ and $D^2_{ym}\bar\Psi$ are bounded independently of $\chi$ and with $D_m\bar\Psi(m_0,x)=D\bar u(0,x)$ where $\bar u$ is defined in \eqref{nchie} and \eqref{nchiesys}.
 \end{Lemma}
 \begin{proof}
Let $\{\phi_n\}_n$ be a countable collection of $C^\infty(\T^d)$ maps such that $\{\phi_n\}_n$ is dense in the set of ${\rm Lip}_1(\T^d)$, which is the set $1$-Lipschitz function on $\T^d$. Then, 
$$
\mathbf{d}(m,m_0)=\sup_{f\in {\rm Lip}_1(\T^d)}\int_{ \T^d}f(x)(m-m_0)(dx)=\sup_{n\in\N}\int_{ \T^d}\phi_n(x)(m-m_0)(dx).
$$
We define $Q:\mathcal P(\T^d)\rightarrow\R$ as follows
\begin{equation}\label{Qdef}
Q(m)=\sum_{n\in\N}\frac{\left(\int_{ \T^d}\phi_n(x)(m-m_0)(dx)\right)^2}{(n+1)^2(\Vert \phi_n\Vert_\infty+\Vert D^2\phi_n\Vert_\infty+1)}.
\end{equation}
The denominator in the above fraction ensures that $Q$ is well defined for any $m\in\mathcal P(\T^d)$. Note that $Q(m)=0$ if and only if, for any $n\in\N$, $\int_{ \T^d}\phi_n(x)(m-m_0)(dx)=0$. In this case, $\mathbf{d}(m,m_0)=0$ and so $m=m_0$. One easily checks that $Q$ is smooth and that its derivatives
\begin{equation}\label{DmQ}
D_m Q(m,y)=2 \sum_{n\in\N}\frac{\left(\int_{ \T^d}\phi_n(x)(m-m_0)(dx)\right) D\phi_n(y)}{(n+1)^2(\Vert \phi_n\Vert_\infty+\Vert D^2\phi_n\Vert_\infty+1)},
\end{equation}
$$
D^2_{my} Q(m,y)=2\sum_{n\in\N}\frac{\left(\int_{ \T^d}\phi_n(x)(m-m_0)(dx)\right) D^2\phi_n(y)}{(n+1)^2(\Vert \phi_n\Vert_\infty+\Vert D^2\phi_n\Vert_\infty+1)},
$$
and
$$
D^2_{mm} Q(m,y,z)=\sum_{n\in\N}\frac{D\phi_n(y)\otimes D\phi_n(z)}{(n+1)^2(\Vert \phi_n\Vert_\infty+\Vert D^2\phi_n\Vert_\infty+1)}
$$
are bounded. Note also that $D_m Q(m_0,y)=0$ for any $y\in\mathcal P(\T^d)$.  We can now define

$$
\bar\Psi(m)=\Psi(m)+Q(m),
$$
where $\Psi$ is the function defined in \eqref{Psi}. By construction, $\bar\Psi$ is such that $\bar\Psi(m)>\chi(m)$ for any $m\neq m_0$ and $\bar\Psi(m_0)=\chi(m_0)$. Moreover,

$$
D_m\bar\Psi(m_0,y)=D_m\Psi(m_0,y)+D_m Q(m_0,y)=D_m\Psi(m_0,y)=D\bar u(0,y).
$$
The boundedness of the derivatives comes from Proposition \ref{Psiconstr} and the properties of $Q$ that we discussed above.
 \end{proof}

\bibliography{../KAMproject/KAMbib}
\bibliographystyle{amsplain}

\end{document}